\documentclass{amsart}
\usepackage{mathrsfs}
\usepackage[all]{xy}

\usepackage{mathtools}
\usepackage{mathbbol}
\usepackage{wasysym}
\usepackage{amssymb}
\usepackage{MnSymbol}
\usepackage{comment}
\usepackage{enumerate}

\usepackage{ushort}

\usepackage{tikz}
\usepackage{caption, subcaption}
\usepackage{tikz-cd}
\usetikzlibrary{calc,decorations.pathreplacing}
\usepackage{tikz}
\usetikzlibrary{calc}
\usetikzlibrary{arrows}
\usetikzlibrary{decorations.pathreplacing}
\usetikzlibrary{intersections}

  \newcommand{\bfa}{\textbf{a}}
  \newcommand{\bfb}{\textbf{b}}

\usepackage{mathrsfs}
\usetikzlibrary{matrix}
\usetikzlibrary{positioning}
\DeclareMathAlphabet{\pazocal}{OMS}{zplm}{m}{n}
\tikzset{>=stealth}

  \usepackage{hyperref}

\usepackage{changepage}
\usepackage{algorithm}
\usepackage[noend]{algpseudocode}
\makeatletter
\def\BState{\, \text{ such that }\, ate\hskip-\ALG@thistlm}
\makeatother

\usepackage{ifpdf}
\ifpdf 
  \usepackage[pdftex]{graphicx}
  \DeclareGraphicsExtensions{.pdf,.png,.jpg,.jpeg,.mps}
  \usepackage{pgf}
\else 
  \usepackage{graphicx}
  \DeclareGraphicsExtensions{.eps,.bmp}
  \usepackage{pgf}
  \usepackage{pstricks}
\fi
\usepackage{epic,bez123}
\usepackage{wrapfig}

\usepackage{soul}
\usepackage{url}

\newtheorem{thm}{Theorem}[section]
\newtheorem{prop}[thm]{Proposition}
\newtheorem{cor}[thm]{Corollary}
\newtheorem{lem}[thm]{Lemma}
\newtheorem{conj}[thm]{Conjecture}
\newtheorem*{claim}{Claim}

\newtheorem{mainthm}{}

\newtheorem{conv}[thm]{Convention}

\theoremstyle{definition}
\newtheorem{defn}[thm]{Definition}

\theoremstyle{remark}

\newtheorem*{rem}{Remark}

\newtheorem*{examples}{Examples}

\newcommand{\isom}{\textrm{Isom}}

\newcommand{\sD}{{\sf D}}   
\newcommand{\sC}{{\sf C}}   
\newcommand{\sR}{{\sf R}}   

\newcommand{\act}{\curvearrowright}

\newcommand{\pG}{\Lambda{G}}

\newcommand{\p}{\mathcal P}
\newcommand{\f}{\mathscr F}

\newcommand{\U}{ X}
\newcommand{\bU}{\overline{  X}}
\newcommand{\hU}{\partial_h{  X}}
\newcommand{\pU}{\partial{  X}}

\newcommand{\dirac}[1]{{\mbox{Dirac}}{(#1)}}

\newcommand{\e}[1]{\omega_{#1}}

\newcommand{\ax}{\mathrm{Ax}}

 \newcommand{\calN}{\mathcal{N}}
 \newcommand{\RR}{\mathbb{R}}
  
\newcommand{\pka}{\partial_\kappa}
\newcommand{\from}{\colon\thinspace}
   \newcommand{\thick}{\ensuremath{\operatorname{thick}}\xspace} 
  \newcommand{\go}{{o}}
  \DeclarePairedDelimiterX{\Norm}[1]{\lVert}{\rVert}{#1}

  \newcommand{\ST}{\,\mathbin{\Big|}\,} 
    
  \newcommand{\param}{{\mathchoice{\mkern1mu\mbox{\raise2.2pt\hbox{$
  \centerdot$}}
  \mkern1mu}{\mkern1mu\mbox{\raise2.2pt\hbox{$\centerdot$}}\mkern1mu}{
  \mkern1.5mu\centerdot\mkern1.5mu}{\mkern1.5mu\centerdot\mkern1.5mu}}}

  \newcommand{\cc}{{\sf c}}

  \newcommand{\mm}{{\sf m}}   
  \newcommand{\nn}{{\sf n}}

  \newcommand{\qq}{{\sf q}}   
  \newcommand{\rr}{{\sf r}}

\newcommand{\proj}{\textbf{d}}

\newcommand{\len }{\ell}

\newcommand{\ywy}[1]{{\color{blue}{#1}}}
\newcommand{\qyl}[1]{{\color{red}{#1}}}
 
\begin{document}

\title[Genericity of sublinearly Morse directions]{Genericity of sublinearly Morse directions in general metric spaces}

\author{Yulan Qing}
\address{Department of Mathematics\\
University of Tennessee at Knoxville\\
1403 Circle Drive, Knoxville, TN 37916, USA}
\email{yqing@utk.edu}

\author{Wenyuan Yang}

\address{Beijing International Center for Mathematical Research\\
Peking University\\
 Beijing 100871, China
P.R.}
\email{wyang@math.pku.edu.cn}
\thanks{(W.Y.) Partially supported by National Key R \& D Program of China (2025YFA1017500) and  National Natural Science Foundation of China (No. 12131009 and No.12326601)}


\subjclass[2000]{Primary 20F65, 20F67}


\dedicatory{}

\keywords{}

\begin{abstract}
In this paper, we show that for a proper statistically convex-cocompact action on a proper geodesic metric space,  the sublinearly Morse boundary has full Patterson-Sullivan measure in the horofunction boundary.  

\end{abstract}

\maketitle
\setcounter{tocdepth}{2} \tableofcontents

\section{Introduction}

Suppose that a group $G$ admits a proper and  isometric action on a proper
geodesic metric space $(\U, d)$. The group $G$ is assumed to be \textit{non-elementary}: it has  no finite index subgroup isomorphic to the integer group $\mathbb Z$ or the trivial group.   

The contracting property captures a key feature of quasi-geodesics in Gromov hyperbolic spaces, rank-1 geodesics in CAT(0) spaces, thick geodesics in Teichm{\"u}ller spaces, and many other settings. In recent years,   this notion and its variants have  proven fruitful in the study of general metric spaces.

Let $A$ be a closed subset of $\U$, and $\pi_A: \U\to A$ be the   nearest-point projection map. We say that $A$ is \textit{$C$--contracting} for $C\ge 0$ if     $$\mathrm{diam}(\pi_A(B))\le C$$ for any metric closed ball $B$ disjoint from $A$.  An element of infinite order is called \textit{contracting}, if it acts by translation on a contracting quasi-geodesic.


The contracting boundary for CAT(0) spaces was introduced by Charney–Sultan \cite{ChaSul} as a quasi-isometric invariant, and has attracted considerable attention in recent years (\cite{Co17,QR22}). It has been observed that the contracting boundary is measurably negligible with respect to harmonic measures and conformal measures.
The  underlying set of the contracting boundary  consists of the endpoints of contracting geodesic rays in the visual boundary.

In contrast, Qing and Rafi \cite{QR22} showed that a much larger class of  sublinearly Morse geodesics equipped with an appropriate topology (called the sublinearly Morse boundary $\pka \U$) is still quasi-isometrically invariant. See Definition \ref{kappaBdryDefn} in Section \ref{sec: sublinear morse bdry}. It is of interest to demonstrate that the sublinearly Morse boundary is indeed larger than its quasi-isometrically  invariant predecessors. In \cite{GQR22}, the authors showed that a subset of the sublinearly Morse directions that are \emph{regularly contracting} (see Definition~\ref{FreqContrDefn}) is generic in the Patterson-Sullivan measure on the visual boundary of CAT(0) spaces under some conditions, and on the Thurston boundary of Teichm{\"u}ller spaces. 

The motivation of this paper is to extend the results of \cite{GQR22} to conformal measure on the horofunction boundary of a general proper metric space. Recently, a general theory of a conformal density has been developed independently by Coulon and the second-named author in \cite{Coulon22, YANG22} on the horofunction boundary in presence of contracting elements. A  key assumption in this generalization is  that the given $G \act X$ is a proper, \emph{statistically convex-cocompact (SCC)} action introduced by the second-named author in \cite{YANG10}, with a contracting element. Excluding negatively curved examples, which lie outside the scope of this paper, we mention the following:
\begin{itemize}
    \item CAT(0)-groups, including most of right-angled Coxeter/Artin groups, which act geometrically on CAT(0) spaces  with {rank-1} elements.
    \item 
    Mapping class groups act on Teichm\"{u}ller space endowed with Teichm\"{u}ller metric.       
\end{itemize}

Using  different and more geometrically flavored tools, we show that sublinearly Morse directions form a full measure set with respect to conformal measures on the horofunction boundary. 

Let $\hU$ be the horofunction boundary. Every sublinearly Morse equivalence class $[\gamma]$ in $\pka \U$ contains a representative given by a geodesic ray $\gamma$,  defining a horofunction $\gamma^+$ in $\pU$. Let $[\gamma^+]$ denote the set of horofunctions up to a finite difference with  $\gamma^+$ (see \textsection \ref{subsec: horofunction boundary} for details). To state the result, we call the union of images $\{[\gamma^+]\subset \hU: [\gamma]\in \pka \U\}$    \emph{sublinearly Morse directions} in $\hU$ (up to taking the finite difference relation). 

\begin{mainthm}[Theorem \ref{HTSThm}]\label{MainThm}
Let $G\act \U$  be a   proper, non-elementary SCC action on a proper geodesic metric space $\U$ with a contracting element. Let $\mu_o$ be the $\e G$--dimensional conformal density on the horofunction boundary $\hU$. Then the set of   sublinearly Morse directions  is a $\mu_o$--full measure subset in $\hU$. 
\end{mainthm}

If $\U$ is a proper CAT(0) space, then there is a unique geodesic ray representative in each sublinearly Morse equivalent class from a fixed basepoint. This gives an injective map from sublinearly Morse boundary $\pka \U$ into the visual boundary (which is canonically homeomorphic to $\hU$ for proper CAT(0) spaces). We thus obtain the following result. This removes certain technical assumptions (e.g. geodesically complete) on $\U$  that were required in a recent result of \cite[Theorem 1.1]{GQR22} to apply the works of Ricks \cite{R17}. 

\begin{mainthm}\label{MainCorollary}
Let $G\act \U$  be a   proper, non-elementary SCC action on a proper CAT(0) space $\U$ with a rank-1 element. Let $\mu_o$ be the $\e G$--dimensional conformal density on the   visual boundary $\hU$. Then  the set of  sublinearly Morse directions is $\mu_o$--full measure. 
\end{mainthm} 

Finally, let us emphasize that \ref{MainThm} holds for any convergence boundary including Thurston boundary of Teichm{\"u}ller spaces. In particular, this gives a different proof of this result in \cite{GQR22}. We refer the reader to Theorem \ref{HTSThm} for  the precise statement. 


\subsection*{Historic background and related works}
In his celebrated proof of what is now called the Mostow rigidity theorem, G. Mostow carried out a two-step strategy as follows: 
\begin{enumerate}
    \item 
    Quasi-isometry between the fundamental groups induces a  quasi-conformal homeomorphism called boundary map between their visual boundaries;
    \item 
    Using the ergodicity of the geodesic flow, the quasi-conformal boundary map is in fact conformal. The key feature in higher dimensions which is absent in dimension 2 is the non‑singularity of the map, meaning it sends a subset of positive Lebesgue measure into the target boundary.
\end{enumerate}
The search for a quasi-isometric invariant boundary has been an active research topic, from a perspective of quasi-isometric classification in geometric group theory. The famous counter-example of Croke-Kleiner \cite{CK00} on the visual boundary of CAT(0) spaces generated  considerable interests  to the subclass of hyperbolic-like directions  as the quasi-isometric invariant boundary, as shown in the aforementioned works of Charney-Sultan \cite{ChaSul} and Qing-Rafi-Tiozzo \cite{QRT2}.

Stimulated by Step 2, it is natural to ask whether the Q.I. invariant boundary under construction is generic in a measure of interest. Lebesgue measure on the visual boundary of rank-1 symmetric spaces is an important instance of conformal measures considered in this paper. First construction of such measures appeared in the seminal work of Patterson \cite{Patt} and has been further developed by Sullivan, setting the stage  for fruitful research in last several decades. Recently, the theory of conformal density on the horofunction boundary is generalized to  groups with contracting element independently  by Coulon \cite{Coulon22} and by the second-named author in \cite{YANG22}. This is the main setup of the paper. We refer the reader to these papers and references therein for a detailed discussion.  

Statistically convex-cocompact actions of groups have received wide attention in the last few years. The concept was introduced in \cite{YANG10} to generalize the idea of convex-cocompact subgroups in Kleinian groups to non-hyperbolic settings. Independently, SCC actions were also introduced by Schapira-Tapie \cite{ST21} under the terminology of strongly positively recurrent actions (or manifolds) in dynamical system. The class of SCC  actions encompasses many examples, but forms a strict subclass of actions with purely exponential growth. In many circumstances, the latter is  equivalent to the so-called positively recurrent actions in \cite{PS18}. 


Another measure of interest is harmonic measure arising from random walk on Poisson boundary.  Recently Qing-Rafi-Tiozzo \cite{QRT2} showed that, when $\kappa(t)=\log t$, the $\kappa$-boundary of the Cayley graph of the mapping class group can be identified with the Poisson boundary of the associated random walks. More general results concerning stationary measures were recently announced by Inhyeok Choi, who in place of our ergodic theoretic and boundary techniques uses a pivoting technique developed by Gou\"ezel. It will be interesting to understand whether, under suitable conditions and when restricted to the corresponding compactification, the hitting measure and Patterson–Sullivan measures supported on sublinearly Morse directions are singular or equivalent.

Let us close the introduction by mentioning certain connection with the following conjecture of Sullivan \cite{Sul}:
\begin{conj}\label{Sullivan}
Suppose that a discrete and isometric action $G\act \mathbb H^n$ is of divergence type. Let $\mu_o$ be the corresponding (unique) Patterson-Sullivan measure supported on $\partial \mathbb H^n$. Then the set of \emph{sublinear limit points} $\xi\in \partial \mathbb H^n$ for which the geodesic ray $\gamma=[o,\xi]$ satisfies 
\begin{align}\label{eq:sublinearlimitset}
\lim_{t\to\infty} \frac{d(\gamma(t), Go) }{t} = 0    
\end{align}
has full $\mu_o$-measure.
\end{conj}
This was confirmed by Sullivan \cite[Corollary 19]{Sul} if the Bowen-Margulis-Sullivan (BMS) measure on the geodesic flow is finite. Note that if the action $G\act\U$ is of divergence type, then $\mu_o$ is supported on the conical limit points of $G$ by Hopf-Tsuji-Sullivan dichotomy.
The conjecture then predicts that, without finiteness of BMS measures,  $\mu_o$ is further supported on the {sublinear limit points}. We refer to \cite{FM20} for relevant discussion. 

Let us explain the relevance  of Conjecture \ref{Sullivan} to this work. 

First, we may  formulate a strengthening  of Conjecture \ref{Sullivan}, and even for groups with contracting elements. Namely, fixing a loxodromic element  $h$ (or  a contracting element  $h\in G$ in general), we replace the whole orbit $Go$  in the equation (\ref{eq:sublinearlimitset}) by the subset $\cup_{g\in G} g\ax(h)$ of orbital points arising from the axis of $h$ up to the $G$-action. We believe this variant of Conjecture \ref{Sullivan} still holds. This is confirmed with finite BM measure on geodesic flow for any nonpositively curved Riemannian manifold in \cite[Thm 5.1]{GQR22}, where it is proved via \textit{ergodic theorems} that regularly contracting rays  are generic. See Definition \ref{FreqContrDefn} and more detail in \cite{GQR22}.

We introduce a group-theoretical version of regularly contracting rays in Definition \ref{FreqBarrierDefn}, which contain \textit{frequently $(r, h)$--barriers}. It follows from the very definition that   the limit in (\ref{eq:sublinearlimitset}) is  zero along such geodesic rays. 
The main bulk of this paper proves that, \textit{if the action is SCC}, then the $\mu_o$-generic geodesic rays have {frequently} $(r, h)$--barriers. The SCC action is sufficient to have finite BMS measure. In this terms,  we  give an ergodic-free proof of \cite[Theorems 5.1, 7.1]{GQR22} in the much larger class of groups with contracting elements.

On the other hand, if  the above strengthening of Sullivan's conjecture holds for the set of regularly contracting rays, then \ref{MainThm} would be valid under a proper action of divergence type (more general than SCC actions). See \cite[Sec. 6]{FM20} for other implication of Sullivan's conjecture. 
\newline
\paragraph{\textbf{Organization of the paper}} 
We present relevant background in Section~\ref{SecPrelim}. Section~\ref{SecDensity} collects all tools from the recently developed conformal density theory on convergence boundary. Section~\ref{SecHTST} is devoted to the proof of \ref{MainThm}.


\section{Preliminaries}\label{SecPrelim}

\subsection{Notations}
Let $(\U, d)$ be a proper geodesic metric space.  Let $A$ be a closed subset of $\U$ and $x$ be a point in $\U$.  By $d(x, A)$ we mean the set-distance 
between $x$ and $A$, i.e. 
\[
d(x, A) : = \inf \big \{ d(x, y): y \in A \big \}. 
\]
Let
\[ \pi_{A}(x) : = \big \{ y\in A: d(x, y) = d(x, A) \big \} \]
be the set of nearest-point projections from $x$ to $A$. Since $\U$ is a proper metric space, 
$\pi_{A}(x)$ is non-empty. We refer to $\pi_{A}(x) $ as the \emph{projection set} of $x$ to $A$. The \textit{projection set} of a subset $B$ to $A$ is similarly defined: $\pi_A(B)=\cup_{x\in B}\pi_A(b)$. Define 
$$
\proj_A(x,y)=\mathrm{diam}(\pi_A(x)\cup\pi_A(y)).
$$ 
If a set $A$ is countable then we use $\sharp A$ to denote the number of elements in $A$.

We write $N_r(A)$ to denote the closed $r$--neighborhood of $A$ for some $r\ge 0$.
We say that a subset $A$ is \textit{$r$--close} to $B$ if $A$ is contained in the $r$--neighborhood of $B$. 


A continuous map $\alpha:I\to X$ from an interval $I\subset (-\infty, +\infty)$ will be referred to as a path. Without explicit mention, the paths we consider are locally rectifiable paths, which means that the restriction  to any finite interval is rectifiable.  

Let $\alpha: [a,b]\to \U$ be a rectifiable  path  with arc-length parametrization from the initial point $\alpha^-=\alpha(a)$ to the terminal point $\alpha^+=\alpha(b)$. If $x=\alpha(s)$ and $y=\alpha(t)$ are two points on $\alpha$, $[x,y]_\alpha$ denotes the parametrized
subpath of $\alpha$ going from $x$ to $y$, that is, the restriction of $\alpha$: $[s,t]\to \U$. 
We also write $\alpha[s,t]$ for the restriction of $\alpha$ to $[s,t]$. 
Let $[x, y]$ denote a geodesic segment (not necessarily unique) between $x, y\in \U$. 

If $\beta$ is a geodesic ray emanating from a point $\go\in X$, then $ \beta({\rr})$ denotes the point on $ \beta $  that is distance $ \rr $ from $\go$.
 
A (locally rectifiable) path $\alpha$ is called a \textit{$(q, Q)$--quasi-geodesic} for constants $q\geq 1$, $Q>0$  if for   any rectifiable subpath $\beta$,
$$\len(\beta)\le q\cdot d(\beta^-, \beta^+)+Q$$
 where   $\ell(\beta)$ denotes the length of $\beta$. Note that if $\beta: [a,b]\subseteq \mathbb R\to \U$ is  parametrized   with arclength parametrization,  for any $s, t \in [a, b]$, we have 
\[
\frac{|s-t|}{q} - Q  \leq d \big(\beta (s), \beta(t)\big) 
   \leq q\, |s-t|+ Q.	
\]
This bi-Lipchitez inequality with bounded error is the usual way to define quasi-geodesics.  
In this paper, we adopt the continuous version of quasi-geodesics as above only for convenience and to make the exposition simpler. 
When $X$ is a geodesic metric space, one can always perturb a quasi-geodesic in an uniform neighborhood to make it continuous 
(see  \cite[Lemma III.1.11]{BriHae}).
 
Denote by $\alpha\cdot \beta$ (or simply $\alpha\beta$) the concatenation of two paths $\alpha, \beta$  provided that $\alpha^+ =
\beta^-$.
 




Let $f, g$ be real-valued functions. Then $f \prec_{c_i} g$ means that
there is a constant $C >0$ depending on parameters $c_i$ such that
$f < Cg$. The symbol $\succ_{c_i}  $ is defined similarly, and  $\asymp_{c_i}$ means both $\prec_{c_i}  $ and $\succ_{c_i}  $  are true. The constant $c_i$ will be omitted if it is a universal constant.

\subsection{Contracting subsets}
\begin{defn}[Contracting subset]
For a given $C\ge 0$, a subset $U$ in $\U$ is called $C$--\textit{contracting}   if for any geodesic $\gamma$ with $d(\gamma,
U) \ge C$, we have
$$\mathrm{diam}(\pi_{U} (\gamma)) \le C.$$
 A
collection of $C$--contracting subsets is referred to
as a $C$--\textit{contracting system}. We shall omit the contracting constant $C$ if its value does not matter.
\end{defn}

Let $\sigma: \mathbb R_{\ge 0}\times \mathbb R_{\ge 0} \to \mathbb R_{\ge 0}$ be a non-negative function.
A subset $U\subseteq\U$ is called $\sigma$-\textit{quasi-convex} if  given $q\ge 1, Q\ge 0$,  any $(q, Q)$--quasi-geodesic with two endpoints in $U$ lies in $N_{R}(U)$ with $R=\sigma(q,Q)$.

\begin{lem}\label{BigThree}
Let $U$ be a $C$-contracting subset for some $C\ge 0$. Then 
\begin{enumerate}
\item
There exists $\sigma$ depending only on $C$ such that   $U$ is $\sigma$--quasi-convex.
\item
For any $r>0$, there exists $\hat C=\hat C(C, r)$ such that a subset $V\subseteq \U$ that is $r$-close to $U$ is $\hat C$--contracting. 
\item
For any $\lambda\ge 1, Q\ge 0$, there exists $\hat C=\hat C(\lambda, Q, C)$ so that any subpath of a  $(\lambda, Q)$--quasi-geodesic $U$ is $\hat C$--contracting. 
\item 
There exists $\hat C=\hat C(C)$ such that   $\proj_U(y, z)\le d(y, z)+\hat C$ for any $y,z\in \U$.
\end{enumerate}
\end{lem}
\begin{proof}
The assertion (3) is proved in \cite[Prop 2.2.3]{YANG11}. The others are straightforward applications of contracting property.   
\end{proof}

In this paper, we are interested in a  contracting system $\f$ with  \textit{$\tau$--bounded intersection} property for a function $\tau: \mathbb R_{\ge 0}\to \mathbb R_{\ge 0}$ if the
following holds
$$\forall U\ne  V\in \f: \quad \mathrm{diam}({N_r (U) \cap N_r (V)}) \le \tau(r)$$
for any $r \geq 0$. This  is, in fact, equivalent to a \textit{bounded projection
property} of $\f$:  there exists a constant $B>0$ such that the
following holds
$$\mathrm{diam}(\pi_{U}(V)) \le B$$
for any $U\ne V \in \f$. See \cite[Lemma 2.3]{YANG6}.

An infinite order element $h \in G$ is called  
\textit{contracting} if the subgroup $\langle h \rangle$ acts by translation on a  contracting quasi-geodesic. By Lemma \ref{BigThree}(2), this is equivalent to saying that the orbital map $n\in \mathbb Z\mapsto h^no\in X$ is a quasi-isometric embedding with contracting image.  The set of contracting elements is preserved under conjugacy. 

\begin{lem}\cite[Lemma 2.11]{YANG10}\label{elementarygroup}
Let $h\in G$ be a contracting element. Then it is contained in the following unique maximal elementary subgroup
$$
E(h)=\{g\in G: \exists n\in \mathbb N_{> 0}, (\;gh^ng^{-1}=h^n)\; \lor\;  (gh^ng^{-1}=h^{-n})\}.
$$
\end{lem}

Keeping in mind the basepoint $o\in\U$, the \textit{axis} of $h$  is defined as the following quasi-geodesic 
\begin{equation}\label{axisdefn}
\ax(h)=\{f o: f\in E(h)\}.
\end{equation} Notice that $\ax(h)=\ax(k)$ and $E(h)=E(k)$    for any contracting element   $k\in E(h)$.

We say that two contracting elements $h,k\in G$ are \textit{independent} if $E(h)\ne E(k)$ and the family of axes $\{g\ax(h), g\ax(k): g\in G\}$ has bounded projection property. Note that this is stronger than the one in literature that only $\ax(h), \ax(k)$ are required to have bounded projection. 

Recall that a group is called \textit{elementary} if it is virtually cyclic.
\begin{lem}\cite[Lemma 4.6]{YANG6}\label{lem:manycontractingelements}
Assume that $G$ acts properly on a proper geodesic metric space $X$ with a contracting element. If $G$ is non-elementary, then $G$ contains infinitely many pairwise independent contracting elements.     
\end{lem}

\subsection{Horofunction boundary}\label{subsec: horofunction boundary}
We  recall the construction of horofunction boundary, which are endowed with the so-called finite and sublinear  difference partitions.

Fix a basepoint $o\in \U$. For  each $y \in  \U$, we define a Lipschitz map $b_y:  \U\to \mathbb R$     by $$\forall x\in \U:\quad b_y(x)=d(x, y)-d(o,y).$$ This   family of $1$--Lipschitz functions sits in the set of continuous functions on $ \U$ vanishing at $o$.  Endowed  with the compact-open topology, the  Arzela-Ascoli Lemma implies that the closure  of $\{b_y: y\in  \U\}$  gives a compactification of $ \U$.  The complement, denoted by $\hU$, of $ \U$ is called  the \textit{horofunction boundary}. Explicitly,  $\hU$ consists of all pointwise limits of the family of functions $\{b_y: y\in  \U\}$: $b_\xi\in \hU$ if for a unbounded sequence of points $y_n\in X$, $b_\xi(x)=\lim_{n\to \infty} b_{y_n}(x)$ holds with every $x\in X$. We write $y_n\to \xi$ and $b_{y_n}\to b_\xi(x)$ accordingly. According to the context, both $\xi$ and $b_\xi$ are used to denote  the boundary points. 

A \textit{Busemann cocycle} $B_\xi:  \U\times \U \to \mathbb R$ (independent of $o$) is given by $$\forall x_1, x_2\in  \U: \quad B_\xi(x_1, x_2)=b_\xi(x_1)-b_\xi(x_2).$$
  
The topological type of horofunction boundary is independent of  the choice of   basepoints. Every isometry $\phi$ of $\U$ induces a homeomorphism on $\bU$:  
$$
\forall y\in \U:\quad b_{\phi\xi}(y):=b_\xi(\phi^{-1}(y))-b_\xi(\phi^{-1}(o)).
$$

\paragraph{\textbf{Finite difference relation}.}
The   \textit{locus} (or $[\cdot]$-class) of   a horofunction  $b_\xi$ with $\xi\in \hU$ consists of all horofunctions $b_\eta$ with $\eta\in \hU$ so that $\|b_\xi-b_\eta\|_\infty<\infty.$  The loci   $[b_\xi]$  of    horofunctions $b_\xi$ form a \textit{finite difference equivalence relation} $[\cdot]$ on $\hU$. The \textit{ $[\cdot]$-saturation} $[\Lambda]$ of a subset $\Lambda\subseteq \hU$ is the union of  $[\cdot]$-classes of all points in $\Lambda$. We say that $\Lambda$ is \textit{saturated} if $[\Lambda]=\Lambda$. 

The horofunction boundary equipped with finite difference relation will be an important example of convergence boundaries, introduced in Subsection \ref{sec convergence bdry}.\newline  


\paragraph{\textbf{Sublinear difference relation}.} It will be also useful to consider another larger relation than the finite difference relation. We say two horofunctions $b_\xi,b_\eta: \U\to [0,\infty)$ have  \textit{sublinear difference} if 
\begin{equation}\label{SublinearDiff}
\lim_{n\to\infty}\sup_{d(o,x)\ge n} \frac{|b_\xi(x)-b_\eta(x)|}{d(o,x)}=0.   
\end{equation}
Since two horofunction representatives of a given point $\xi\in \pU$ differ by a constant for different basepoints, this relation is independent on the choice of basepoint. The sublinear difference relation is an equivalence relation. We denote by $[\xi]_s$ the  equivalent class of $\xi\in \pU$, and $[\hU]_s$ the resulting quotient space of $\hU$. It is clear that $[\xi]\subseteq [\xi]_s$, so $[\hU]$ is also quotient of $[\hU]_s$.  

Any geodesic ray $\alpha$ tends to a unique point, denoted by $\alpha^+$, at the horofunction boundary $\hU$. Namely, the associated horofunction is as follows
\begin{equation}\label{defn:busemanfunc}
\forall x\in X: \quad \alpha^+(x):=\lim_{t\to\infty} [d(\alpha(t),x)-t]    
\end{equation}
Let us emphasize that the notation $\alpha^+$  means either a boundary point or the horofunction associated to the point. 

\begin{lem}\label{lineardiverging}
Let $\alpha,\beta:[0,\infty)\to \U$ be two geodesic ways emanating from the same basepoint $o\in \U$, ending at  $\alpha^+,\beta^+\in \hU$ respectively. 
If there are two sequences of real numbers $s_n, r_n\to\infty$ satisfying  
\begin{align}
\label{rnsn} \limsup_{n\to\infty} r_n/s_n<1\\
\label{distrn}\forall n\gg0, \; \limsup_{t\to\infty} d(o, [\alpha(s_n), \beta(t)])\le r_n
\end{align}
then $[\alpha^+]_s\ne [\beta^+]_s$.   
\end{lem}
\begin{proof}
Let $x_n=\alpha(s_n)\in \alpha$ so we have $\alpha^+(x_n)=-s_n$ as defined in (\ref{defn:busemanfunc}). We are going to prove that the sequence $x_n$ violates the inequality (\ref{SublinearDiff}). 
For small enough $\epsilon>0$, according to (\ref{distrn}), there are $t_m\to \infty$ so that $$d(o, [\alpha(s_n), \beta(t_m)])\le r_n+\epsilon$$ so by triangle inequality, $$|d(\alpha(s_n),\beta(t_m))- s_n-t_m|\le 2(r_n+\epsilon)$$ Thus, letting $s_n\to\infty$ and $\epsilon\to 0$, 
$$\begin{aligned}
\limsup_{s_n\to\infty} \frac{|\beta^+(x_n)-\alpha^+(x_n)|}{s_n}&=\limsup_{s_n\to\infty} \frac{|\beta^+(x_n)+s_n|}{s_n}\\
&=\frac{|\lim_{t_m\to\infty} [d(\beta(t_m),x_n)-t_m-s_n] +2s_n|}{s_n} > 0    
\end{aligned}$$
where the last inequality uses (\ref{rnsn}). Hence, $[\alpha^+]_s\ne [\beta^+]_s$ follows by definition.   
\end{proof}

\begin{lem}\label{SublinearClassLem}
For any $C,r>0$, there exists $L_0=L_0(C,r)$ with the following property. Let $\alpha,\beta$ be two geodesic rays from the same basepoint ending at $[\alpha^+]\ne [\beta^+]$ respectively. If $\alpha$ contains infinitely many disjoint $C$--contracting segments of length $L$ in its $r$--neighborhood for some $L\ge L_0$, then  $[\alpha^+]_s\ne [\beta^+]_s$.    
\end{lem}
\begin{proof}
Pick up two unbounded sequences of $x_n\in \alpha$ and $y_n\in \beta$. Let $\{p_m: m\ge 1\}$ be the sequence of infinitely many $C$--contracting segments in $N_r(\alpha)$. For each $p_m$, if    $[x_n,y_n]\cap N_C(p_m)= \emptyset$ hold for infinitely many $n\gg 1$, then  $d(o,[x_n,y_n])\le D=d(o,p_m)+r+C$.  It thus follows from Lemma \ref{lineardiverging} that  $[\xi]_s\ne [\eta]_s$. So let us assume that    $[x_n,y_n]\cap N_C(p_m)\ne  \emptyset$  for all large $n\gg 1$.  Passing to  subsequence, we may further assume that $[x_n,y_n]\cap N_C(p_m)= \emptyset$ for each $n,m\ge 1$. The reminder of the proof seeks for contradiction.

We next prove that  $\beta\cap N_C(p_m)= \emptyset$ for all large $m$. Indeed, if  $\beta\cap N_C(p_m)\ne  \emptyset$ holds for infinitely many $m$ then the intersection of $\alpha$ and $N_{r+C}(\beta)$ gets unbounded since  $p_m$ is escaping and $p_m\subset N_r(\alpha)$.  By definition of horofunctions, this implies $[\alpha^+]= [\beta^+]$ in the same finite difference class, contradicting the assumption.  

In summary, we proved that for given $p_m$, the $C$-neighborhood of $p_m$ is disjoint with $\beta$ and $[x_n,y_n]$ for all $n\ge 1$.  We now project $\beta$ and $[x_n,y_n]$ to $p_m$, so $\proj_{p_m}(x_n,y_n),\; \mathrm{diam}(\pi_{p_m}(\beta))\le C$ by the $C$--contracting property. Further, we project the subpath of $\alpha$ before $p_m$ and after $p_m$ as in the proof of Lemma \ref{close barriers in two geodesics}. We would show that $\len(p_m)$ is bounded from above by a constant $L_0$ depending on  $r, C$; the details are left to the interested reader. This contradicts the choice of $\len(p_m)=L>L_0$. The proof is complete.    
\end{proof}

For hyperbolic spaces, the sublinear and finite difference relations turn out to give the same partition on $\hU$. 
\begin{cor}
If $\U$ is a Gromov hyperbolic geodesic space, then $[\hU]_s$ is homeomorphic to the Gromov boundary $\pU$.   
\end{cor}
\begin{proof}
Any geodesic ray in a hyperbolic space is uniformly contracting. The conclusion follows immediately from Lemma \ref{SublinearClassLem}.
\end{proof}

\subsection{Sublinearly Morse boundaries}\label{sec: sublinear morse bdry}

Now we introduce a large class of quasi-geodesic rays that are quasi-isometry invariant. Intuitively, these quasi-geodesics have a weak Morse-like property. To begin with, we fix a function that is sublinear in the following sense:
\newline 
\paragraph{\textbf{Sublinear functions}}
We fix a function 
\[
\kappa \from [0,\infty) \to [1,\infty)
\] 
that is monotone increasing, concave and sublinear, that is
\[
\lim_{t \to \infty} \frac{\kappa(t)} t = 0. 
\]
Note that using concavity, for any $a>1$, we have
\begin{equation} \label{Eq:Concave}
\kappa(a t) \leq a \left( \frac 1a \, \kappa (a t) + \left(1- \frac 1a\right) \kappa(0) \right) 
\leq a \, \kappa(t).
\end{equation}

We say a quantity $\sD$ \emph{is small compared to a radius $\rr>0$} if 
\begin{equation} \label{Eq:Small} 
\sD \leq \frac{\rr}{2\kappa(\rr)}. 
\end{equation} 

\begin{rem}
The assumption that $\kappa$ is increasing and concave makes certain arguments
cleaner, otherwise they are not really needed. One can always replace any 
sublinear function $\kappa$, with another sublinear function $\overline \kappa$
so that $\kappa(t) \leq \overline \kappa(t) \leq \sC \, \kappa(t)$ for some constant $\sC$ 
and $\overline \kappa$ is monotone increasing and concave. For example, define 
\[
\overline \kappa(t) = \sup \Big\{ \lambda \kappa(u) + (1-\lambda) \kappa(v) \ST 
\ 0 \leq \lambda \leq 1, \ u,v>0, \ \text{and}\ \lambda u + (1-\lambda)v =t \Big\}.
\]
The requirement $\kappa(t) \geq 1$ is there to remove additive errors in the definition
of $\kappa$--contracting geodesics. 
\end{rem}

\begin{defn}[$\kappa$--neighborhood]  \label{Def:Neighborhood} 
Recall that, for $x \in X$, we have $\Norm{x} = d(o, x)$. To simplify notation, we often drop $\Norm{\param}$. That is, for $x \in X$, we define
\[
\kappa(x) := \kappa(\Norm{x}). 
\]
For a closed set $Z$ and a constant $\nn$, define the $(\kappa, \nn)$--neighbourhood 
of $Z$ to be 
\[
\calN_\kappa(Z, \nn) = \Big\{ x \in \U \ST 
  d(x, Z) \leq  \nn \cdot \kappa(x)  \Big\}.
\]

\end{defn} 

\begin{figure}[h!]
\begin{tikzpicture}
 \tikzstyle{vertex} =[circle,draw,fill=black,thick, inner sep=0pt,minimum size=.5 mm] 
[thick, 
    scale=1,
    vertex/.style={circle,draw,fill=black,thick,
                   inner sep=0pt,minimum size= .5 mm},
                  
      trans/.style={thick,->, shorten >=6pt,shorten <=6pt,>=stealth},
   ]

 \node[vertex] (a) at (0,0) {};
 \node at (-0.2,0) {$o$};
 \node (b) at (10, 0) {};
 \node at (10.6, 0) {$\gamma$};
 \node (c) at (6.7, 2) {};
 \node[vertex] (d) at (6.68,2) {};
 \node at (6.7, 2.4){$x$};
 \node[vertex] (e) at (6.68,0) {};
 \node at (6.7, -0.5){$x_{b}$};
 \draw [-,dashed](d)--(e);
 \draw [-,dashed](a)--(d);
 \draw [decorate,decoration={brace,amplitude=10pt},xshift=0pt,yshift=0pt]
  (6.7,2) -- (6.7,0)  node [black,midway,xshift=0pt,yshift=0pt] {};

 \node at (7.8, 1.2){$\nn \cdot \kappa(x)$};
 \node at (3.6, 0.7){$\Norm x$};
 \draw [thick, ->](a)--(b);
 \path[thick, densely dotted](0,0.5) edge [bend left=12] (c);
\node at (1.4, 1.9){$(\kappa, \nn)$--neighbourhood of $\tau$};
\end{tikzpicture}
\caption{A $\kappa$--neighbourhood of a geodesic ray $\tau$ with multiplicative constant $\nn$.}
\end{figure}

In this paper, $Z$ is either a geodesic or a quasi-geodesic. That is, we can write 
$\calN_{\kappa}(\tau, \nn)$ to mean the $(\kappa, \nn)$--neighborhood of the image of 
the geodesic ray $\tau$. Or, we can use phrases like 
``the quasi-geodesic $\beta$ is $\kappa$--contracting" or 
``the geodesic $\gamma$ is in a $(\kappa, \nn)$--neighbourhood of the geodesic $c$".

We recall the definition of $\kappa$--contracting and $\kappa$--Morse sets from \cite{QRT2}.

\begin{defn}[$\kappa$--Morse I] \label{Def:W-Morse} 
We say a closed subset $Z$ of $\U$ is \emph{$\kappa$--Morse} in the first sense if there is a function
\[
\mm_Z \from \mathbb R_{\ge 0}\times \mathbb R_{\ge 0} \to \mathbb R_{\ge 0}
\]
so that if $\beta \from [s,t] \to \U$ is a $(q, Q)$--quasi-geodesic with endpoints 
on $Z$ then
\[
\beta[s,t]  \subseteq\calN_{\kappa} \big(Z,  \mm_Z(q, Q)\big). 
\]
We refer to $\mm_{Z}$ as a \emph{Morse gauge} for $Z$. Without loss of generality one can assume
\begin{equation}
\mm_Z(q, Q) \geq \max(q, Q). 
\end{equation} 
\end{defn}

\begin{defn}[$\kappa$--Morse II] \label{Def:S-Morse} 
We say a closed subset $Z$ of $\U$ is \emph{$\kappa$--Morse} in the second sense if there is a function 
$\mm_Z\from \RR^2 \to \RR$ such that, for every constants $\rr>0$, $\nn>0$ and every
sublinear function $\kappa'$, there is an $\sR= \sR(Z, \rr, \nn, \kappa')>0$ where the 
following holds: Let $\eta \from [0, \infty) \to \U$ be a $(q, Q)$--quasi-geodesic ray 
so that $\mm_Z(q, Q)$ is small compared to $\rr$, let $t_\rr$ be the first time 
$\Norm{\eta(t_\rr)} = \rr$ and let $t_\sR$ be the first time $\Norm{\eta(t_\sR)} = \sR$. Then
\[
d\big(\eta(t_\sR), Z\big) \leq \nn \cdot \kappa'(\sR)
\quad\Longrightarrow\quad
\eta[0, t_\rr] \subseteq\calN_{\kappa}\big(Z, \mm_Z(q, Q)\big). 
\]
\end{defn} 
It is also natural to generalize the notion of contracting to the sublinear setting: 
\begin{defn}[$\kappa$--contracting] \label{Def:kappa-Contracting}
For a closed subspace $Z$ of $\U$, we say $Z$ is \emph{$\kappa$--contracting} if there 
is a constant $\cc_Z$ so that, for every $x,y \in \U$
\[
d(x, y) \leq d( x, Z) \quad \Longrightarrow \quad
\proj_Z \big( x,  y \big) \leq \cc_Z \cdot \kappa(x).
\]
\end{defn} 
\begin{figure}[h!]
\begin{tikzpicture}[scale=0.4]
 \tikzstyle{vertex} =[circle,draw,fill=black,thick, inner sep=0pt,minimum size=.5 mm]
 
[thick, 
    scale=1,
    vertex/.style={circle,draw,fill=black,thick,
                   inner sep=0pt,minimum size= .5 mm},
                  
      trans/.style={thick,->, shorten >=6pt,shorten <=6pt,>=stealth},
   ]
   
    \node[vertex] (a) at (0,0) [label=$\go$]  {};
    \node(b) at (10, 0) {};
    \draw(a)--(b){};
    \draw (1,1) circle (0.5);
     \draw (4,1.5) circle (1);
      \draw (8,2) circle (1.5);
      
      \draw [-, dash dot] (0.5, 1) to [bend left=20] (0.5+0.3, 0);
       \draw[-, dash dot] (1.5, 1) to [bend right=20](1.5-0.3, 0);
       
        \draw[-, blue, very thick](0.8, 0) to (1.2, 0);
       
        \draw [-, dash dot](3, 1.5)to [bend left=20](3+0.6, 0);
        \draw [-, dash dot](5,1.5)to [bend right=20](5-0.6,0);
        
        \draw[-, blue, very thick](3.6, 0) to (4.4, 0);
        
        \draw[-, dash dot](6.5, 2)to [bend left=20](6.5+0.9, 0);
        \draw[-, dash dot](9.5, 2)to [bend right=20](9.5-0.9, 0);
        
        \draw[-, blue, very thick](7.4, 0) to (8.6, 0);
   
\end{tikzpicture}
\caption{A sublinearly contracting geodesic ray}
\end{figure}

The following theorem summarizes the relation among the above three notions.
\begin{thm}[\cite{QRT2}]
Let  $\U$ be a proper geodesic metric space and let $\gamma$ be a quasi-geodesic ray in $\U$. 
Then 
\begin{enumerate}
\item $\gamma$ is $\kappa$--Morse in the first sense  if and only if $\gamma$ is $\kappa$--Morse in the second sense. 
\item If $\gamma$ is $\kappa$--contracting then $\gamma$ is 
$\kappa$--Morse.
\item If $\gamma$ is $\kappa$-Morse then it is $\kappa'$-contracting.
\end{enumerate}
\end{thm}
In the sequel, we say a set is \textit{$\kappa$--Morse} if it is either $\kappa$--Morse  in  either sense.

Analogous to Lemma 2.3 (2), $\kappa$--Morse  is a property that can also be established by approximity. That is to say,  if a quasi-geodesic ray  $\alpha$ sublinearly fellow travel a $\kappa$--Morse ray $\beta$, then $\alpha$ is also $\kappa$--Morse. We now make this precise. First assume without loss of generality that a quasi-geodesic ray is a continuous path. Define 
\begin{align*}
\alpha_r :=
  \{ \alpha(t_0) \,|\, &\alpha(t_0) \in (\alpha \cap X \backslash B(o, r) ),\text{ and for any other }\\
&\text{} \alpha(t)\in \alpha \cap X \backslash B(o, r), \text { we have } t_0 \leq t \}\\
\end{align*}
We say two quasi-geodesic rays \emph{sublinearly fellow travel} if 
\[ \lim_{r \to \infty} \frac{d(\alpha_r, \beta_r)}{r} = 0.\]
\begin{thm}\cite{QRT2}
Let $\alpha$ be a $(q_1, Q_2)$--quasi-geodesic ray and let $\beta$ be a $(q_2, Q_2)$--quasi-geodesic ray that is $\kappa$--Morse. If $\alpha$ and $\beta$ sublinearly fellow travel, then there exists a $\kappa$--neighborhood depending only on $(q_1, Q_1)$ and $(q_2, Q_2)$ such that 
\[
\alpha \in N_\kappa \big (\beta, m((q_1, Q_1), (q_2, Q_2))\big).\]
Furthermore, $\alpha$ is a $\kappa$--Morse ray.
\end{thm}

Lastly, a quasi-geodesic is called \emph{sublinearly Morse} if it is $\kappa$--Morse for some sublinearly growing function $\kappa$.

\begin{defn}[Sublinearly Morse boundary]\label{kappaBdryDefn}
Given a sublinear function $\kappa$, let $\partial_{\kappa}X$ denote the set of equivalence classes of $\kappa$--Morse quasi-geodesics. Equipped with a coarse version of cone topology, we call this set the \emph{$\kappa$--Morse boundary} of $\U$ and denote it $\pka \U$ (for more details, see \cite{QRT2}). 
\end{defn}
It is shown in \cite{QRT2} that $\U\cup \partial_{\kappa}\U$ with the coarse cone topology is a QI-invariant space and a metrizable topological space.

\subsection{Convergence boundary}\label{sec convergence bdry}
In this subsection, we discuss an axiomatic approach introduced in \cite{YANG22} to the boundary of proper geodesic metric spaces in presence of contracting subsets. The motivating examples are the Gromov boundary of hyperbolic spaces and visual boundary of CAT(0) spaces.   To  describe  further examples in \ref{ConvbdryExamples}, we need to introduce some terminology.

Let $(\U, d)$ be a proper metric space admitting an isometric action of a non-elementary countable group $\Gamma$ with a contracting element. Consider a metrizable compactification $\bU:=\pU\cup \U$, so that $\U$ is open and dense in $\bU$. We also assume that the action of $\textrm{Isom}(\U)$ extends by homeomorphism to the boundary  $\pU$. 

We   equip $\pU$    with a  $\isom(\U)$--invariant  partition $[\cdot]$:   $[\xi]=[\eta]$ implies $[g\xi]=[g\eta]$ for any $g\in \isom(\U)$.   The  \textit{locus} $[Z]$  of a subset $Z\subseteq\pU$ is the union of all $[\cdot]$--classes of $\xi\in Z$.  We say that $\xi$ is \textit{minimal} if $[\xi]=\{\xi\}$, and a subset $U$ is \textit{$[\cdot]$--saturated} if $U=[U]$.

We say that $[\cdot]$ restricts to be a \textit{closed} partition on a $[\cdot]$--saturated subset $U\subseteq \pU$ if  $x_n\in U\to \xi\in \pU$ and $y_n\in U\to\eta\in \pU$ are two sequences with $[x_n]=[ y_n]$, then $[\xi]=[\eta]$. (Possibly $\xi, \eta$ are not in $U$ anymore.) If $U=\pU$, this is equivalent to say that the relation $\{(\xi,\eta): [\xi]=[\eta]\}$ is a closed subset in $\pU\times \pU$, so the quotient space $[\pU]$ is Hausdorff. In general, $[\cdot]$ may  be not closed over the whole $\pU$.

For completeness, let us fix a few terminologies in general topology. 
Let $x_n$ be a sequence of points in $\U$. We say that $x_n$ \textit{converges} (or \textit{tends}) to a point $\xi\in \pU$  if any open neighborhood of $\xi$ in $\U\cup \pU$ contains all but finitely many $x_n$. Accordingly, $\xi$ is called the \textit{limit point} of $x_n$ and  $x_n$ is  a \textit{convergent} sequence. Given a subset $U$ in $X$, the limit points of all convergent sequences in  $U$ form \textit{accumulation points} of $U$.

With respect to the given partition, we say that  $x_n\in X$ \textit{accumulates}  to $[\xi]$ if the  accumulate points of the set $\{x_n: n\ge 1\}$ are contained in $[\xi]$. This implies that the sequence of $[\cdot]$-classes $[x_n]$ tends to the $[\cdot]$-class $[\xi]$ in the quotient space $[\pU]$. By abusing language, we also say that $x_n$ \textit{tends} to $[\xi]$. Further,  an infinite ray $\gamma$ \textit{terminates} at a point in $[\xi] \in \pU$ if any unbounded sequence of points on $\gamma$ accumulates in $[\xi]$. 

Let $A$ be a $C$-contracting subset in $\U$ for some $C>0$. Define the \textit{strong cone} of a subset $A$  at $o$ as follows $$\Omega_o(A)=\{x\in \U: \mathrm{diam}([o,x]\cap A)\ge 10C\}$$ A sequence of subsets $A_n$ is called \textit{escaping} if $d(o,A_n)\to \infty$ for some (or any) $o\in \U$.

\begin{defn}\label{ConvBdryDefn}
    
We say that $(\pU,[\cdot])$ is a \textit{convergence compactification} of $X$ if the following assumptions hold.
\begin{enumerate}
    \item[\textbf{(A)}]Any contracting geodesic ray $\gamma$ accumulates to a $[\cdot]$-class in $\pU$, denoted as $[\gamma^+]$, that is a closed subset. Moreover, any sequence of    $y_n\in \U$ with escaping projections $\pi_\gamma(y_n)$ accumulates to $[\gamma^+]$. 

    \item[\textbf{(B)}]
    Let $\{\gamma_n\subseteq \U :n\ge 1\}$ be an escaping sequence of $C$--contracting  quasi-geodesics for some $C>0$. Then for any given $o\in \U$, there exist an escaping subsequence of $\{\gamma_n\cup \Omega_o(\gamma_n): n\ge 1\}$ denoted by $A_n$  and a $[\cdot]$-class of some point $\xi\in \pU$ such that $A_n$  accumulates    into $[\xi]$:
    
    Any convergent sequence of points $x_n\in A_n$ tends to a point in $[\xi]$.
    \item[\textbf{(C)}]
    The set   $\mathcal C$ of \textit{non-pinched} points  $\xi\in \pU$ is non-empty. If $x_n, y_n\in \U$ are two sequences of points converging to $[\xi]$, then $[x_n,y_n]$ is an escaping sequence of geodesic segments.  
\end{enumerate}  
\end{defn}  
\begin{figure}
    \centering

\tikzset{every picture/.style={line width=0.75pt}} 

\begin{tikzpicture}[x=0.75pt,y=0.75pt,yscale=-1,xscale=1]

\draw    (27.5,123) -- (217.5,123) ;
\draw [shift={(219.5,123)}, rotate = 180] [color={rgb, 255:red, 0; green, 0; blue, 0 }  ][line width=0.75]    (10.93,-3.29) .. controls (6.95,-1.4) and (3.31,-0.3) .. (0,0) .. controls (3.31,0.3) and (6.95,1.4) .. (10.93,3.29)   ;
\draw [shift={(27.5,123)}, rotate = 0] [color={rgb, 255:red, 0; green, 0; blue, 0 }  ][fill={rgb, 255:red, 0; green, 0; blue, 0 }  ][line width=0.75]      (0, 0) circle [x radius= 3.35, y radius= 3.35]   ;
\draw    (146.5,78) .. controls (135.94,102) and (140.13,107.57) .. (144.02,121.25) ;
\draw [shift={(144.5,123)}, rotate = 255.07] [color={rgb, 255:red, 0; green, 0; blue, 0 }  ][line width=0.75]    (10.93,-3.29) .. controls (6.95,-1.4) and (3.31,-0.3) .. (0,0) .. controls (3.31,0.3) and (6.95,1.4) .. (10.93,3.29)   ;
\draw [shift={(146.5,78)}, rotate = 113.75] [color={rgb, 255:red, 0; green, 0; blue, 0 }  ][fill={rgb, 255:red, 0; green, 0; blue, 0 }  ][line width=0.75]      (0, 0) circle [x radius= 3.35, y radius= 3.35]   ;
\draw   (393,102) .. controls (393,90.95) and (408.68,82) .. (428.01,82.01) .. controls (447.34,82.01) and (463,90.97) .. (463,102.01) .. controls (463,113.06) and (447.33,122.01) .. (428,122.01) .. controls (408.67,122) and (393,113.05) .. (393,102) -- cycle ;
\draw   (310,107.98) .. controls (310,96.93) and (325.67,87.98) .. (345,87.99) .. controls (364.33,87.99) and (380,96.95) .. (380,107.99) .. controls (380,119.04) and (364.33,127.99) .. (345,127.99) .. controls (325.67,127.98) and (310,119.03) .. (310,107.98) -- cycle ;
\draw    (257.5,110.97) .. controls (307.5,113.98) and (376.5,118.99) .. (411.5,132) ;
\draw [shift={(411.5,132)}, rotate = 20.39] [color={rgb, 255:red, 0; green, 0; blue, 0 }  ][fill={rgb, 255:red, 0; green, 0; blue, 0 }  ][line width=0.75]      (0, 0) circle [x radius= 3.35, y radius= 3.35]   ;
\draw    (257.5,110.97) .. controls (296.5,102.98) and (433.5,89) .. (478.5,78) ;
\draw [shift={(478.5,78)}, rotate = 346.26] [color={rgb, 255:red, 0; green, 0; blue, 0 }  ][fill={rgb, 255:red, 0; green, 0; blue, 0 }  ][line width=0.75]      (0, 0) circle [x radius= 3.35, y radius= 3.35]   ;
\draw [shift={(257.5,110.97)}, rotate = 348.42] [color={rgb, 255:red, 0; green, 0; blue, 0 }  ][fill={rgb, 255:red, 0; green, 0; blue, 0 }  ][line width=0.75]      (0, 0) circle [x radius= 3.35, y radius= 3.35]   ;
\draw  [dash pattern={on 4.5pt off 4.5pt}]  (478.5,78) .. controls (508.04,93.74) and (489.09,100.79) .. (522.92,106.73) ;
\draw [shift={(524.5,107)}, rotate = 189.46] [color={rgb, 255:red, 0; green, 0; blue, 0 }  ][line width=0.75]    (10.93,-3.29) .. controls (6.95,-1.4) and (3.31,-0.3) .. (0,0) .. controls (3.31,0.3) and (6.95,1.4) .. (10.93,3.29)   ;
\draw [shift={(478.5,78)}, rotate = 28.05] [color={rgb, 255:red, 0; green, 0; blue, 0 }  ][fill={rgb, 255:red, 0; green, 0; blue, 0 }  ][line width=0.75]      (0, 0) circle [x radius= 3.35, y radius= 3.35]   ;
\draw  [dash pattern={on 4.5pt off 4.5pt}]  (146.5,78) .. controls (176.19,93.82) and (150.03,95.96) .. (212.58,113.46) ;
\draw [shift={(214.5,114)}, rotate = 195.48] [color={rgb, 255:red, 0; green, 0; blue, 0 }  ][line width=0.75]    (10.93,-3.29) .. controls (6.95,-1.4) and (3.31,-0.3) .. (0,0) .. controls (3.31,0.3) and (6.95,1.4) .. (10.93,3.29)   ;
\draw [shift={(146.5,78)}, rotate = 28.05] [color={rgb, 255:red, 0; green, 0; blue, 0 }  ][fill={rgb, 255:red, 0; green, 0; blue, 0 }  ][line width=0.75]      (0, 0) circle [x radius= 3.35, y radius= 3.35]   ;
\draw    (202.5,267) .. controls (180.5,254) and (181.5,222) .. (203.5,208) ;
\draw [shift={(203.5,208)}, rotate = 327.53] [color={rgb, 255:red, 0; green, 0; blue, 0 }  ][fill={rgb, 255:red, 0; green, 0; blue, 0 }  ][line width=0.75]      (0, 0) circle [x radius= 3.35, y radius= 3.35]   ;
\draw [shift={(202.5,267)}, rotate = 210.58] [color={rgb, 255:red, 0; green, 0; blue, 0 }  ][fill={rgb, 255:red, 0; green, 0; blue, 0 }  ][line width=0.75]      (0, 0) circle [x radius= 3.35, y radius= 3.35]   ;
\draw    (86.12,240) -- (183.5,240) ;
\draw [shift={(185.5,240)}, rotate = 180] [color={rgb, 255:red, 0; green, 0; blue, 0 }  ][line width=0.75]    (10.93,-3.29) .. controls (6.95,-1.4) and (3.31,-0.3) .. (0,0) .. controls (3.31,0.3) and (6.95,1.4) .. (10.93,3.29)   ;
\draw [shift={(86.12,240)}, rotate = 0] [color={rgb, 255:red, 0; green, 0; blue, 0 }  ][fill={rgb, 255:red, 0; green, 0; blue, 0 }  ][line width=0.75]      (0, 0) circle [x radius= 3.35, y radius= 3.35]   ;
\draw  [dash pattern={on 4.5pt off 4.5pt}]  (230.5,209) .. controls (268.92,217.87) and (253,232.55) .. (294.56,233.95) ;
\draw [shift={(296.5,234)}, rotate = 181.3] [color={rgb, 255:red, 0; green, 0; blue, 0 }  ][line width=0.75]    (10.93,-3.29) .. controls (6.95,-1.4) and (3.31,-0.3) .. (0,0) .. controls (3.31,0.3) and (6.95,1.4) .. (10.93,3.29)   ;
\draw  [dash pattern={on 4.5pt off 4.5pt}]  (234.5,269) .. controls (258.14,262.11) and (253.64,241.63) .. (295.55,239.1) ;
\draw [shift={(297.5,239)}, rotate = 177.4] [color={rgb, 255:red, 0; green, 0; blue, 0 }  ][line width=0.75]    (10.93,-3.29) .. controls (6.95,-1.4) and (3.31,-0.3) .. (0,0) .. controls (3.31,0.3) and (6.95,1.4) .. (10.93,3.29)   ;

\draw (122,131.4) node [anchor=north west][inner sep=0.75pt]    {$\pi _{\gamma }( y_{n})$};
\draw (160,61.4) node [anchor=north west][inner sep=0.75pt]    {$y_{n}$};
\draw (222,112.4) node [anchor=north west][inner sep=0.75pt]    {$[ \xi ]$};
\draw (21,127.4) node [anchor=north west][inner sep=0.75pt]    {$o$};
\draw (337,133.39) node [anchor=north west][inner sep=0.75pt]  [rotate=-0.01]  {$\gamma _{1}$};
\draw (527,97.42) node [anchor=north west][inner sep=0.75pt]  [rotate=-0.01]  {$[ \xi ]$};
\draw (416,89.4) node [anchor=north west][inner sep=0.75pt]  [rotate=-0.01]  {$\gamma _{n}$};
\draw (408,138.39) node [anchor=north west][inner sep=0.75pt]  [rotate=-0.01]  {$y_{1} \in \Omega _{o}( \gamma _{1})$};
\draw (443,48.4) node [anchor=north west][inner sep=0.75pt]    {$y_{n} \in \Omega _{o}( \gamma _{n})$};
\draw (253.5,91.37) node [anchor=north west][inner sep=0.75pt]    {$o$};
\draw (68.57,230.4) node [anchor=north west][inner sep=0.75pt]    {$o$};
\draw (212,200.4) node [anchor=north west][inner sep=0.75pt]    {$x_{n}$};
\draw (213,260.4) node [anchor=north west][inner sep=0.75pt]    {$y_{n}$};
\draw (301,227.4) node [anchor=north west][inner sep=0.75pt]    {$[ \xi ] \subseteq \mathcal{C}$};
\draw (370,228.4) node [anchor=north west][inner sep=0.75pt]    {$\Longrightarrow $};
\draw (421,224.4) node [anchor=north west][inner sep=0.75pt]    {$d( o,[ x_{n} ,y_{n}])\rightarrow \infty $};
\draw (18,10) node [anchor=north west][inner sep=0.75pt]   [align=left] {(\textbf{A})};
\draw (251,10) node [anchor=north west][inner sep=0.75pt]   [align=left] {(\textbf{B})};
\draw (20,232) node [anchor=north west][inner sep=0.75pt]   [align=left] {(\textbf{C})};

\end{tikzpicture}
    \caption{Illustrate Assumptions (A)(B)(C) in Definition \ref{ConvBdryDefn}}
    \label{fig:convbdry}
\end{figure}
\begin{examples}\label{ConvbdryExamples}
The first three convergence boundaries below are equipped with  a \textit{maximal} partition $[\cdot]$ (i.e. each $[\cdot]$-class is singleton). See \cite{YANG22} for more details.
    \begin{enumerate}
        \item 
        Hyperbolic space $\U$ with Gromov boundary $\pU$, where  all boundary points are non-pinched.
        \item 
        CAT(0) space $\U$ with visual boundary $\pU$ (homeomorphic to horofunction boundary), where all boundary points are non-pinched.
        \item
        The Cayley graph $\U$ of a relatively hyperbolic group with Bowditch or Floyd boundary $\pU$, where  conical limit points are non-pinched.

        If $\U$ is infinitely ended, we could also take $\pU$ as the end boundary with the same statement.
        \item
        Teichm\"{u}ller space $\U$ with Thurston  boundary $\pU$, where $[\cdot]$ is given by Kaimanovich-Masur partition \cite{KaMasur} and uniquely ergodic   points are non-pinched.
        \item
        Any proper metric space $\U$ with horofunction boundary $\pU$, where $[\cdot]$ is given by finite  difference partition and all boundary points are non-pinched.
        If $\U$ is the cubical CAT(0) space, the horofunction boundary is exactly the Roller boundary. 
        If $\U$ is the Teichm\"{u}ller space with Teichm\"{u}ller metric, the horofunction boundary is  the Gardiner-Masur  boundary (\cite{LS14,Walsh19}). 
    \end{enumerate}
\end{examples}
From these examples, it appears that a fixed proper metric space can possess multiple distinct convergence boundaries. In applications, the horofunction boundary provides  a meaningful and non-trivial convergence boundary for \textit{any} proper action.
\begin{thm}\label{HorobdryConvergence}\cite[Theorem 1.1]{YANG22}
Assume that $\U$ is a proper metric space which contains at least one contracting geodesic ray. Then  the horofunction boundary of $X$ is a convergence boundary with finite difference relation $[\cdot]$, where all boundary points are non-pinched. 
\end{thm}
As finite difference partition is finer than the sublinear difference partition, we see by definition that Theorem \ref{HorobdryConvergence} holds for sublinear difference partition as well.

\subsection{Regularly  contracting geodesic rays}


We begin by recalling the notion of a regularly contracting segment from \cite{GQR22}. A related  notion involving group actions will be given in Definition \ref{FreqBarrierDefn}. 

For any $\theta\in (0,1]$, if $\gamma$ is a geodesic, a \textit{$\theta$--segment} means   a connected and closed subsegment of $\gamma$ with length $\theta \ell(\gamma)$.
 
\begin{defn}\label{FreqContrDefn}
Fix $r,C>0$ and $L>0, \theta\in (0,1]$.
\begin{enumerate}

\item A geodesic segment $\gamma$  is \textit{$(r, C, L)$--contracting at $\theta$–frequency}     if  every $\theta$--segment of $\gamma$ contains a subsegment of length $L$ that is $r$--close to a  $C$–contracting geodesic. More precisely,
for any $0 < t < (1-\theta)\ell(\gamma)$ there is an interval 
of times $[s-L/2, s+L/2] \subset [t, t+ \theta \ell(\gamma)]$ and a $C$--contracting geodesic $p$ such that, 
\begin{equation}\label{rclosetopEQ}
t \in [s-L/2 , s+L/2]
\qquad\Longrightarrow \qquad
d(\gamma(t), p) \leq r.    
\end{equation}

\item A geodesic ray $\gamma$  is \textit{$(r, C, L)$--contracting at $\theta$–frequency} if there is an $R_0>0$ such that any initial segment of $\gamma$ with length at least $R_0$ (i.e. the segment $\gamma[0,t]$ for any $t\ge R_0$) is $(r, C, L)$--contracting at $\theta$–frequency.  

\item A geodesic ray $\gamma$ is \emph{$(r,C)$--regularly contracting} 
if $\gamma$ is $(r, C, L)$--contracting at $\theta$–frequency  for each $L>0$ and each $\theta\in (0,1]$.   Further, $\gamma$ is \textit{regularly contracting} if it is {$(r,C)$--regularly contracting} for some $r>0,C>0$.  
    
\end{enumerate}
\end{defn}

Let us recall the following result which is crucial in this study.
\begin{thm} \cite{GQR22}\label{GQRtheorem}
If a geodesic ray $\gamma$ is $(r,C)$--regularly contracting for some $r,C>0$, then it is $\kappa$--contracting for some sublinear function
$\kappa$ depending on $r,C$. In particular, it is also $\kappa$--Morse. 
\end{thm} 

We next make the connection with convergence boundary. Namely, regularly contracting geodesic rays accumulate into a $[\cdot]$-class in the convergence boundary.
\begin{lem}\label{RegCConverge}
Let $(\pU, [\cdot])$ be a convergence compactification of $X$. Given $\theta, r, C, L>0$, let $\gamma$ be an $(r, C, L)$--contracting geodesic ray  at $\theta$--frequency. Then for any  large $L>10(C+r)$, $\gamma$  accumulates into a $[\cdot]$--class denoted by $[\gamma^+]$  in $\pU$.
\end{lem}
\begin{proof}
By the definition of frequently contracting rays, there exists  a sequence of the $C$--contracting geodesic segments $p_n$  satisfying the property (\ref{rclosetopEQ}). By excising subsegments  from $p_n$ with length $L$, we may further assume that $p_n$ is an escaping sequence of $C$--contracting geodesic of length $L$. 

Denote  $\beta_n:=N_r(p_n)$. One verifies that $\beta_n$ is $(C+r)$-contracting, and by  (\ref{rclosetopEQ}), $\gamma$ intersects $\beta_n$ in a diameter at least $\len(p_n)=L$. So if $L>10(C+r)$, then applying  Definition \ref{ConvBdryDefn}(B) to the escaping sequence $\{\beta_n\}$, we thus obtain a  $[\cdot]$--class  in  $\pU$,  denoted by $[\gamma^+]$, and $\Omega_o(\beta_n)$ accumulates to $[\gamma^+]$. Noting that since $\gamma$ intersects each $\beta_n$  in diameter at least  $L$,   any sequence of points on $\gamma$ are eventually contained in $\Omega_o(\beta_n)$. Hence, $\gamma$ accumulates into $[\gamma^+]$ by Assumption (B).  The proof is complete.    
\end{proof}

Let us equip the space $\U$ of interest with a convergence boundary $(\pU,[\cdot])$ (see Examples \ref{ConvbdryExamples}). A universal choice is to consider the horofunction boundary endowed with finite or sublinear difference partition.

By Lemma \ref{RegCConverge}, for given $L\gg 10(r+ C)$, we denote by $\mathcal{FC}(\theta,r, C,L)$  the set of  all $[\cdot]$--classes   $[\xi]\subset \pU$ so that some geodesic ray $\gamma$ in $\U$ starting at $o$ and ending at $[\xi]$  is $(r,C,L)$--contracting at $\theta$--frequency. 

Further, let $\mathcal{RC}(r, C)$ denote the set of all $[\cdot]$--classes   $[\xi]\subset \pU$ so that some geodesic ray $\gamma$ in $\U$ starting at $o$ and ending at $[\xi]$  is $(r,C)$--regular contracting. 

The relation between these sets shall be crucial in further discussion. 
\begin{prop}\label{RCasCountableIntersection}
For any $r\ge C>0$, let $\hat r, L_0$ be the constants given below by Lemma \ref{RegContrRayClass}. Then the following holds
$$
\begin{aligned}\label{RCCountIntEQ}
\mathcal{RC}( r, C) \subseteq \bigcap_{L_0\le L\in \mathbb N}\left (\bigcap_{\theta\in (0,1]\cap \mathbb Q} \mathcal{FC}(\theta,r, C,L)\right)  \subseteq \mathcal{RC}(\hat r, C)   
\end{aligned}
$$
where $\mathbb Q$ denotes the set of rational numbers.
\end{prop}
The first inclusion is immediate and the reminder of this subsection is devoted to the second one. 
We first start with some elementary lemmas  used later on.


\begin{lem}\label{close barriers in two geodesics}
For any $r\ge C,M>0$ there exist $\hat r=\hat r(r,C,M), L_0=L_0(r,C,M)>0$ with the following property.

Assume that $\alpha$ and $\beta$ are two geodesics with $d(\alpha^-,\beta^-), d(\alpha^+,\beta^+)\le M$. Let $p$ be a $C$-contracting segment with length at least $L_0$. If  $d(p^-,\alpha),d(p^+,\alpha)\le r$, then $d(p^-,\beta),d(p^+,\beta)\le \hat r$.    
\end{lem}
\begin{proof}
Let $u,v\in \alpha$ be the entry and exit points of $\alpha$ in $N_r(p)$.  Since $d(p^-,\alpha),d(p^+,\alpha)\le r$, we have $d(u,v)\ge \len(p)-2r$. Take points $u',v'\in p$ with $d(u,u'),d(v,v')\le r$. This implies $\len(p)\ge d(u',v')\ge \len(p)-4r$, which yields $d(u',p_-)+ d(v',p_+)\le 4r$.

We first prove that $\beta$  intersects $N_r(p)$. 
As the shortest projection to  $p$ is a coarsely $1$--Lipschitz map by Lemma \ref{BigThree}(4), the projection of a ball with radius $\max\{r,M\}$  to $p$  has diameter at most $B$ depending on $r, M$ and $C$. Thus, $$\max\{\proj_{p}(\alpha^-,\beta^-),\;\proj_{p}(\alpha^+,\beta^+),\;\proj_{p}(u,u'),\;\proj_{p}(v,v')\}\le B.$$ Now, if $\beta$ was disjoint with  $N_r(p)$, we would project $\beta$ to $p$ with diameter at most $C$. Each projection of $[\alpha^-,u]$ and $[v,\alpha^-]$ to $p$ gets bounded by $C$, so we obtain 
$$\begin{aligned}
\len(p)&\le 4r+d(u',v')\\
&\le 4r+ \proj_{p}(\alpha^-,\beta^-)+\proj_{p}(\alpha^+,\beta^+)+\proj_{p}(u,u')+\proj_{p}(v,v')+\\
&\;\;+\mathrm{diam}(\pi_{p}(\beta))+\proj_{p}(\alpha^-,u)+\proj_{p}(v,\alpha^-)\le 4r+4B+3C    
\end{aligned}$$ Since $\len(p) \ge L_0$ by hypothesis, if  we choose $$L_0>4r+4B+3C$$ which gives a contradiction, then $\beta$ has to intersect $N_r(p)$. 

To conclude the proof, let $x$ and $y$ be  the entry and exit points of $\beta$ in $N_r(p)$, so $d(x,p),d(y,p)\le r$ and $d(x,y)\le \len(p)+2r$. Arguing similarly as above we project $[\alpha^-,\beta^-]$, $[\alpha^-,u]$ and $[\beta^-,x]$ to $p$, which yields $d(x,u)\le 2B+2C=:D$. By symmetry, we also have $d(y,v)\le D$. 

Recalling  $d(u,v)\ge \len(p)-2r$, we see  $d(x,y)\ge d(u,v)-2D\ge \len(p)-2D-2r$. This implies that $d(x,p^-),d(y,p^+)\le 2D+4r$. Setting $\hat r=2D+4r$ concludes the proof of the lemma.   
\end{proof}

Before stating the next result, we briefly explain its motivation.
Note that the space $X$ is not necessarily uniquely geodesic, and the $[\cdot]$-equivalence classes under consideration may therefore be non-singleton. Consequently, a boundary class need not admit a canonical geodesic ray representative, and the parameters associated with frequent contraction may vary among different representatives of the same $[\cdot]$-class. The result below provides effective control of these parameters when passing between such geodesic rays.

\begin{lem}\label{RegContrRayClass}
Given constants $r\ge C>0$, there exist $\hat r=\hat r(r,C), L_0=L_0(r,C)$ with the following property.

Given $L\ge L_0$ and $\theta\in (0,1]$, let $\gamma$  be an $(r, C, L)$--contracting  geodesic ray  at $\theta$–frequency starting at a point $o\in \U$. Assume that the terminal $[\cdot]$-class $[\gamma^+]$ of  $\gamma$ is non-pinched. Then any geodesic ray from $o$ ending at $[\gamma^+]$ in $\pU$ is $(\hat r, C, L)$--contracting   at $\theta$–frequency.
\end{lem}
\begin{proof}
Throughout the proof, all constants depend only on $r$ and $C$. 
By assumption, let  $p_n$ be a sequence of the $C$--contracting geodesic segments  for $\gamma$ witnessing the property (\ref{rclosetopEQ}). Excising subsegments from $p_n$ with length $L$ if necessary, we may assume that $\{p_n: n\ge 1\}$ is    escaping and each $p_n$ has length  $L$. 

Let $x_n,y_n\in \gamma$ be the entry and exit points of $\gamma$ in $N_r(p_n)$, so the endpoints of $p_n$ being contained in $N_r(\gamma)$ yields $$d(x_n,y_n)\ge \ell(p_n)-2r\ge L-2r$$ Let $x_n', y_n'\in p_n$ so that $d(x_n,x_n'),d(y_n,y_n')\le r$. The first paragraph in the proof of Lemma \ref{close barriers in two geodesics} proves $d(x_n',p_n^-),d(y_n',p_n^+)\le 4r$. That is, the projection of  $x_n$ (resp. $y_n$) to $p_n$ lies in the $4r$-neighborhood of $p_n^-$ (resp. $p_n^+$). By the $C$-contracting property, $[o,x_n]_\gamma$ projects to $p_n$  as a subset with diameter at most $C$, which thus is  $(C+4r)$-close to $p_n^-$. Similarly, if $u$ is a point  on $\gamma$ with $d(o,u)\ge d(o,y_n)$, the projection of $[y_n,u]_\gamma$ to $p_n$ is $(C+4r)$-close to $p_n^+$. 

\begin{figure}
    \centering

\tikzset{every picture/.style={line width=0.75pt}} 

\begin{tikzpicture}[x=0.75pt,y=0.75pt,yscale=-1,xscale=1]

\draw    (100,121) .. controls (139.6,80.41) and (355.62,76.08) .. (438.53,80.85) ;
\draw [shift={(441,81)}, rotate = 183.53] [fill={rgb, 255:red, 0; green, 0; blue, 0 }  ][line width=0.08]  [draw opacity=0] (8.93,-4.29) -- (0,0) -- (8.93,4.29) -- cycle    ;
\draw    (100,121) .. controls (150.49,137.83) and (379.36,137.02) .. (442.15,109.83) ;
\draw [shift={(444,109)}, rotate = 154.98] [fill={rgb, 255:red, 0; green, 0; blue, 0 }  ][line width=0.08]  [draw opacity=0] (8.93,-4.29) -- (0,0) -- (8.93,4.29) -- cycle    ;
\draw [shift={(100,121)}, rotate = 18.43] [color={rgb, 255:red, 0; green, 0; blue, 0 }  ][fill={rgb, 255:red, 0; green, 0; blue, 0 }  ][line width=0.75]      (0, 0) circle [x radius= 3.35, y radius= 3.35]   ;
\draw    (378,80) -- (382,123) ;
\draw [shift={(382,123)}, rotate = 84.69] [color={rgb, 255:red, 0; green, 0; blue, 0 }  ][fill={rgb, 255:red, 0; green, 0; blue, 0 }  ][line width=0.75]      (0, 0) circle [x radius= 3.35, y radius= 3.35]   ;
\draw [shift={(378,80)}, rotate = 84.69] [color={rgb, 255:red, 0; green, 0; blue, 0 }  ][fill={rgb, 255:red, 0; green, 0; blue, 0 }  ][line width=0.75]      (0, 0) circle [x radius= 3.35, y radius= 3.35]   ;
\draw [line width=1.5]  [dash pattern={on 5.63pt off 4.5pt}]  (196.98,110.34) -- (248.98,108.34) ;
\draw [shift={(248.98,108.34)}, rotate = 357.8] [color={rgb, 255:red, 0; green, 0; blue, 0 }  ][fill={rgb, 255:red, 0; green, 0; blue, 0 }  ][line width=1.5]      (0, 0) circle [x radius= 4.36, y radius= 4.36]   ;
\draw [shift={(196.98,110.34)}, rotate = 357.8] [color={rgb, 255:red, 0; green, 0; blue, 0 }  ][fill={rgb, 255:red, 0; green, 0; blue, 0 }  ][line width=1.5]      (0, 0) circle [x radius= 4.36, y radius= 4.36]   ;
\draw  [dash pattern={on 4.5pt off 4.5pt}] (187.92,93.61) .. controls (187.69,87.34) and (192.58,82.07) .. (198.84,81.84) -- (245,80.12) .. controls (251.26,79.89) and (256.53,84.78) .. (256.76,91.05) -- (258.03,125.08) .. controls (258.26,131.35) and (253.37,136.61) .. (247.11,136.85) -- (200.95,138.56) .. controls (194.69,138.8) and (189.42,133.91) .. (189.19,127.64) -- cycle ;
\draw  [line width=6] [line join = round][line cap = round] (187.8,91.22) .. controls (187.8,91.22) and (187.8,91.22) .. (187.8,91.22) ;
\draw  [line width=6] [line join = round][line cap = round] (254.8,83.22) .. controls (254.8,83.22) and (254.8,83.22) .. (254.8,83.22) ;
\draw    (126.1,106.29) .. controls (147.37,96.29) and (184.97,90.02) .. (222.58,86.19) .. controls (265.6,81.8) and (308.63,80.59) .. (327,80.5)(127.38,109) .. controls (148.43,99.11) and (185.66,92.96) .. (222.88,89.17) .. controls (265.79,84.8) and (308.7,83.59) .. (327.01,83.5) ;
\draw [shift={(329,82)}, rotate = 0] [color={rgb, 255:red, 0; green, 0; blue, 0 }  ][line width=0.75]      (0, 0) circle [x radius= 3.35, y radius= 3.35]   ;
\draw [shift={(124,109)}, rotate = 332.4] [color={rgb, 255:red, 0; green, 0; blue, 0 }  ][line width=0.75]      (0, 0) circle [x radius= 3.35, y radius= 3.35]   ;
\draw    (131.69,123.82) .. controls (155.32,126.66) and (188.22,128.27) .. (221.12,128.9) .. controls (231.86,129.1) and (242.6,129.2) .. (253.02,129.2) .. controls (284.37,129.2) and (312.81,128.31) .. (329.53,126.74)(131.33,126.8) .. controls (155.04,129.64) and (188.05,131.27) .. (221.06,131.9) .. controls (231.82,132.1) and (242.58,132.2) .. (253.02,132.2) .. controls (284.49,132.2) and (313.03,131.31) .. (329.81,129.73) ;
\draw [shift={(332,128)}, rotate = 353.99] [color={rgb, 255:red, 0; green, 0; blue, 0 }  ][line width=0.75]      (0, 0) circle [x radius= 3.35, y radius= 3.35]   ;
\draw [shift={(129,125)}, rotate = 7.25] [color={rgb, 255:red, 0; green, 0; blue, 0 }  ][line width=0.75]      (0, 0) circle [x radius= 3.35, y radius= 3.35]   ;

\draw (288,63.4) node [anchor=north west][inner sep=0.75pt]    {$\beta _{2} \subset \gamma $};
\draw (292,134.4) node [anchor=north west][inner sep=0.75pt]    {$\beta _{1} \subset \alpha $};
\draw (370,64.4) node [anchor=north west][inner sep=0.75pt]    {$u$};
\draw (379,128.4) node [anchor=north west][inner sep=0.75pt]    {$v$};
\draw (82,115.4) node [anchor=north west][inner sep=0.75pt]    {$o$};
\draw (217,110.4) node [anchor=north west][inner sep=0.75pt]    {$p_{n}$};
\draw (166,74.4) node [anchor=north west][inner sep=0.75pt]    {$x_{n}$};
\draw (255,64.4) node [anchor=north west][inner sep=0.75pt]    {$y_{n}$};
\draw (202.95,141.96) node [anchor=north west][inner sep=0.75pt]    {$N_{r}( p_{n})$};
\draw (436,83.4) node [anchor=north west][inner sep=0.75pt]    {$\left[ \gamma ^{+}\right] =\left[ \alpha ^{+}\right]$};
\draw (168,101.4) node [anchor=north west][inner sep=0.75pt]    {$p_{n}^{-}$};
\draw (261,96.4) node [anchor=north west][inner sep=0.75pt]    {$p_{n}^{+}$};

\end{tikzpicture}
    \caption{The proof of Lemma \ref{RegContrRayClass}. The dotted area indicates the $r$-neighborhood of $p_n$, and the double lines $\beta_1$ and $\beta_2$ are $\theta$-segments of $[o,u]_\gamma$ and $[o,v]_\alpha$ accordingly.}
    \label{fig:placeholder}
\end{figure} 
Towards the proof, let $\alpha$ be any geodesic ray from $o$ and ending at $[\gamma^+]$, which is a non-pinched class. 
\begin{claim}
For each $n\ge 1$, $\alpha$ intersects $N_{C}(p_n)$, and $d(p_n^-,\alpha)\le  M:=3C+4r$. 
\end{claim}
\begin{proof}[Proof of the claim]
Let us fix $p_n$ for some $n\ge 1$, and take  two convergent sequences of points  $u_m\in \gamma$ and $v_m\in \alpha$ so that $\{u_m\}$ and $\{v_m\}$ both tend to  $[\gamma^+]$. Since $[\gamma^+]$ is non-pinched,   the assumption (C) in the convergence boundary implies the sequence of segments $[u_m,v_m]$ is escaping. So for large $m\gg 1$, by the $C$-contracting property of $p_n$, $[u_m,v_m]$ projects to a subset on $p_n$ with diameter at most $C$. 

Let us assume to the contrary that $\alpha$ is disjoint with $N_{C}(p_n)$. By the $C$-contracting property, $\alpha$   projects to a subset of $p_n$ with diameter at most $C$.   According to the first paragraph, we see $$
\begin{aligned}
\ell(p_n) &\le \mathrm{diam}(\pi_{p_n}([o,x_n]_\gamma))+\mathrm{diam}(\pi_{p_n}([y_n,u_m]_\gamma))\\
&\;\;+\proj_{p_n}(u_m,v_m)+\mathrm{diam}(\pi_{p_n}(\alpha))  \\
&\le 2(C+4r)+C+C=4(C+2r)    
\end{aligned}$$  Choosing $L_0>4(C+2r)$  contradicts the assumption that $\ell(p_n)\ge L\ge L_0$. Hence, $\alpha$ intersects $N_{C}(p_n)$ for each $n\ge 1$. 

By a similar projection argument and  looking at the entry point of $\alpha$ in $N_{C}(p_n)$, we  derive $d(p_n^-,\alpha)\le (C+4r)+ 2C\le 3C+4r$.    
\end{proof}
By the claim, we then have $d(p_n^-,\alpha),d(p_n^-,\gamma)\le M$. Since $\{p_n:n\ge 1\}$ is an {escaping} sequence, $\alpha$ is $2M$-close to  $\gamma$ at an unbounded sequence of points. Let $\hat r=\hat r(r,2M,C)$  and $L_0=L_0(r,2M,C)$ further  satisfy the conclusion of Lemma \ref{close barriers in two geodesics}.  

To finish the proof, it is sufficient to verify that $\alpha$ is $(\hat r, C, L)$--contracting   at $\theta$–frequency. 
Namely, given any $v\in \alpha$, let $\beta_1$ be a   $\theta$-segment of $[o,u]_\alpha$. We need show that some segment of length $L$  in $\beta_1$ is $r$-close to a $C$-contracting segment. 

Indeed, let $u\in \gamma$  so that $d(o,u)=d(o,v)$.    Let $\beta_2$ be the $\theta$-segment of $[o,u]_\gamma$ with the same location as $\beta_1$ in $[o,v]_\alpha$; that is $d(o,\beta_1^-)=d(o,\beta_2^-)$ and $\ell(\beta_1)=\ell(\beta_2)$.  Since $\gamma$ is $( r, C, L)$--contracting   at $\theta$–frequency, there is a $C$-contracting segment $q$ that is $r$-close to  a segment with length $L$ of $\beta_2$. By the above discussion, since $\alpha$ is $2M$-close to  $\gamma$ at an unbounded sequence of points,  $[o,u]_\gamma$ is contained in a segment $[o,u']$ of $\gamma$ so that $u'$ is $2M$-close to $v'\in \alpha$.    

We are now ready to apply  Lemma \ref{close barriers in two geodesics} to the pair of $[o,u']$ and $[o,v']$. Since $q$ is $r$-close to a segment with length $L$ in a subsegment $\beta_2$ of  $[o,u]_\gamma$, we see that  $q^-,q^+$ are in the $\hat r$-neighborhood of $[o,v']$. Recall that $\beta_1$ in $[o,v]_\alpha$ has the same location as $\beta_2$ in $[o,u]_\gamma$. Combined with the fact that the two endpoints of $q$ are contained in $N_{\hat r}([o,u]_\gamma)$ and $N_{\hat r}([o,v]_\alpha)$, we derive that $d(q^-,\beta_1),d(q^+,\beta_1)\le \hat r$ up to increasing a multiple of $\hat r$. Therefore,  $\alpha$ is $(\hat r, C, L)$--contracting   at $\theta$–frequency, concluding the proof.
\end{proof}

We are ready to complete the proof of Proposition \ref{RCasCountableIntersection}.

\begin{proof}[Proof of Proposition \ref{RCasCountableIntersection}]
By definition, $\mathcal{RC}(r, C)$ is clearly a subset of the set in the middle term of the above equation.  For the second inclusion, let $[\xi]$ be a $[\cdot]$-class in the middle set. That is, for any $L>L_0 $ and $ \theta\in (0,1]\cap \mathbb Q$, there exists a geodesic ray denoted as $\gamma_{L,\theta}$ depending on $L,\theta$ and ending at $[\xi]$ so that $\gamma_{L,\theta}$ is $(r,C,L)$-contracting at frequency $\theta$. Let us fix a geodesic ray $\gamma$ ending at $[\xi]$. By Lemma \ref{RegContrRayClass}, since  $\gamma_{L,\theta}$ is $(r,C,L)$-contracting at frequency $\theta$,   $\gamma$ is $(\hat r,C,L)$-contracting at frequency $\theta$.  Hence, $\gamma$ is a $(\hat r,C)$-contracting ray, so $[\xi]$ is contained in $\mathcal{RC}(\hat r, C)$.
\end{proof}

\subsection{Statistically convex-cocompact actions}

In this subsection, we first recall a class of statistically convex-cocompact actions introduced in \cite{YANG10}. 

Let $X$ be a proper geodesic metric space on which a non-elementary (i.e. not virtually cyclic) group $G$  acts properly by isometry.
Consider the ball of radius $n$ centered at $o\in \U$: 
\begin{equation}\label{BallEQ}
N(o, n)=\{g\in G: d(o, go)\le n\}
\end{equation}    
Fix $\Delta\ge 1$. Consider the annulus of radius $n$ centered at $o\in \U$: 
\begin{equation}\label{AnnulusEQ}
A(o, n, \Delta)=\{g\in G: |d(o, go)-n|\le\Delta\}
\end{equation}
The \emph{critical exponent} of a subset $\Gamma$ of $G$ is defined as
\[
\e \Gamma =\limsup_{n \to \infty} \frac{ \log {\sharp (N(o, n) \cap \Gamma)}}{n}.
\]
By Lemma \ref{lem:manycontractingelements}, $G$ contains two independent loxodromic elements. By a standard ping-pong argument, one deduces that $G$ contains a non-abelian free group. Thus, it follows that $\e G>0$.
We say that the proper action $G\act X$ has \textit{purely exponential growth} if $\e G<\infty$ and $\sharp N(o,n)\asymp \mathrm{e}^{n\e G}$ for any $n\ge 1$.

\begin{defn}
A subset $A$ of the group $G$ is called \textit{negligible} if 
$$\begin{aligned}
\frac{\sharp A\cap  N(o,n)}{\sharp N(o,n)} \to 0, \text{ as } n\to\infty.   
\end{aligned}$$
It is said to be \textit{exponentially negligible} if there exists $\lambda\in (0,1)$ and $c>0$ so that 
$$\begin{aligned}
\forall n\ge 1,\;\;\; \frac{\sharp A\cap  N(o,n)}{\sharp N(o,n)} \le c\lambda^n.   
\end{aligned}$$
\end{defn}
\begin{rem}
If the proper action $G\act X$ has {purely exponential growth}, then the exponential negligibility of $A$ is equivalent to the strict inequality $\e A<\e G$. Such sets are sometimes called growth tight in literature.    
\end{rem}

Given constants $0\leq M_1\leq M_2$, let $\mathcal{O}_{M_1,M_2}$ be the set of elements $g\in G$ such that there exists some geodesic $\gamma$ between $N_{M_2}(o)$ and $N_{M_2}(go)$ with the property that the interior of $\gamma$ lies outside $N_{M_1}(G o)$.
\begin{figure}[h]
    \centering

\tikzset{every picture/.style={line width=0.75pt}} 

\begin{tikzpicture}[x=0.75pt,y=0.75pt,yscale=-1,xscale=1]

\draw    (133.5,87) -- (357.5,102.86) ;
\draw [shift={(359.5,103)}, rotate = 184.05] [color={rgb, 255:red, 0; green, 0; blue, 0 }  ][line width=0.75]    (10.93,-3.29) .. controls (6.95,-1.4) and (3.31,-0.3) .. (0,0) .. controls (3.31,0.3) and (6.95,1.4) .. (10.93,3.29)   ;
\draw [shift={(133.5,87)}, rotate = 4.05] [color={rgb, 255:red, 0; green, 0; blue, 0 }  ][fill={rgb, 255:red, 0; green, 0; blue, 0 }  ][line width=0.75]      (0, 0) circle [x radius= 3.35, y radius= 3.35]   ;
\draw   (73.5,96) .. controls (90.5,55) and (206.5,165) .. (281.5,156) .. controls (356.5,147) and (390.5,77) .. (418.5,120) .. controls (446.5,163) and (339,198) .. (248,193) .. controls (157,188) and (56.5,137) .. (73.5,96) -- cycle ;
\draw    (400.5,120) -- (367.24,105.22) ;
\draw [shift={(364.5,104)}, rotate = 23.96] [fill={rgb, 255:red, 0; green, 0; blue, 0 }  ][line width=0.08]  [draw opacity=0] (8.93,-4.29) -- (0,0) -- (8.93,4.29) -- cycle    ;
\draw [shift={(400.5,120)}, rotate = 203.96] [color={rgb, 255:red, 0; green, 0; blue, 0 }  ][fill={rgb, 255:red, 0; green, 0; blue, 0 }  ][line width=0.75]      (0, 0) circle [x radius= 3.35, y radius= 3.35]   ;
\draw   (60.5,105.5) .. controls (60.5,84.51) and (77.51,67.5) .. (98.5,67.5) .. controls (119.49,67.5) and (136.5,84.51) .. (136.5,105.5) .. controls (136.5,126.49) and (119.49,143.5) .. (98.5,143.5) .. controls (77.51,143.5) and (60.5,126.49) .. (60.5,105.5) -- cycle ;
\draw   (358,121.5) .. controls (358,100.79) and (374.79,84) .. (395.5,84) .. controls (416.21,84) and (433,100.79) .. (433,121.5) .. controls (433,142.21) and (416.21,159) .. (395.5,159) .. controls (374.79,159) and (358,142.21) .. (358,121.5) -- cycle ;
\draw    (98.5,105.5) -- (130.85,88.4) ;
\draw [shift={(133.5,87)}, rotate = 152.14] [fill={rgb, 255:red, 0; green, 0; blue, 0 }  ][line width=0.08]  [draw opacity=0] (8.93,-4.29) -- (0,0) -- (8.93,4.29) -- cycle    ;
\draw [shift={(98.5,105.5)}, rotate = 332.14] [color={rgb, 255:red, 0; green, 0; blue, 0 }  ][fill={rgb, 255:red, 0; green, 0; blue, 0 }  ][line width=0.75]      (0, 0) circle [x radius= 3.35, y radius= 3.35]   ;
\draw  [line width=6] [line join = round][line cap = round] (363.5,103) .. controls (363.5,103) and (363.5,103) .. (363.5,103) ;

\draw (233,162.4) node [anchor=north west][inner sep=0.75pt]    {$N_{M_{1}}( Go)$};
\draw (133,57.4) node [anchor=north west][inner sep=0.75pt]    {$B( o,M_{2})$};
\draw (81,98.4) node [anchor=north west][inner sep=0.75pt]    {$o$};
\draw (375,60.4) node [anchor=north west][inner sep=0.75pt]    {$B( go,M_{2})$};
\draw (397.5,124.9) node [anchor=north west][inner sep=0.75pt]    {$go$};
\draw (193,98.4) node [anchor=north west][inner sep=0.75pt]    {$\gamma \cap N_{M_{1}}( Go) =\emptyset $};

\end{tikzpicture}
    \caption{Illustration of $\mathcal O_{M_1,M_2}$.}
    \label{fig:enter-label}
\end{figure}

\begin{defn}[SCC action]\label{SCCDefn}
If there exist positive constants $M_1,M_2>0$ such that $\e {\mathcal{O}_{M_1,M_2}}<\e G<\infty$, then the proper action of $G$ on $\U$ is called \textit{statistically convex-cocompact} (SCC).
\end{defn}
\begin{rem}
The  definition of  $\mathcal{O}_{M_1,M_2}$ is motivated by considering  the action of  the fundamental group of a finite volume negatively curved Hadamard manifold on its universal cover. It is then easy to see that for appropriate constants $M_1, M_2>0$, the set $\mathcal{O}_{M_1,M_2}$ coincides with  the union of the orbits of cusp subgroups up to a  finite Hausdorff distance. The assumption in SCC actions was called the \textit{parabolic gap condition} by Dal'bo, Otal and Peign\'{e} in \cite{DOP}.  
\end{rem}
We have chosen two parameters $M_1,M_2$ so that the definition of a statistically convex-cocompact action  is flexible and easy to verify. It is enough to take $M_1=M_2=M$ in our use. Henceforth, we set $\mathcal{O}_M:=\mathcal{O}_{M,M}$ for ease of notation.

When the SCC action contains a contracting element, the definition is independent of the basepoint (see \cite[Lemma 6.2]{YANG10}).

A key notion in studying growth gaps is that of a barrier-free element introduced in \cite{YANG10}.

\begin{defn}\label{BarrierDef}
Fix constants $r,M>0$.
\begin{enumerate}
\item
Given $r>0$ and $f\in G$, we say that a geodesic $\gamma$ contains an $(r,f)$--\textit{barrier} if there exists an element $t\in G$ so that
    $$\begin{aligned} \label{eq:barrier}
    \max\{d(t\cdot o,\gamma),d(t\cdot fo,\gamma)\}\leq r.
    \end{aligned}$$
    By abusing language, we also say that $to$ is an $(r,f)$--{barrier} of $\gamma$.
    If no such $t\in G$ exists so that the above inequalities hold, then $\gamma$ is called $(r,f)$--\textit{barrier-free}. 
    Further, we say that $\gamma$ is  $(r,F)$--\textit{barrier-free} for a subset $F\subset G$ if $\gamma$ is  $(r,f)$--\textit{barrier-free} for some $f\in F$.

\item
An element $g\in G$ is $(r,M,f)$--\textit{barrier-free} if there exists an $(r,f)$--barrier-free geodesic between $N_M(o)$ and $N_M(go)$. Similarly, we define $(r,M,F)$--\textit{barrier-free} elements for a subset $F\subset G$.
\end{enumerate}
\end{defn}

Given $r,M>0$ and  any $f\in G$, let $\mathcal{V}_{r,M,f}$ be the set of all $(r,M,f)$--barrier-free elements in $G$. The following results will be crucially used in the next sections. 

\begin{prop}\label{pro:growthtight}\cite[Theorems B \& C]{YANG10}
Suppose that a non-elementary group $G$ admits a proper action on a proper geodesic space $(\U,d)$ with a contracting element.   Then for any large enough $M\gg 0$, there exists $r=r(M)>0$ with the following property.
\begin{enumerate}
\item  
If $G$ has purely exponential growth,  then $\mathcal{V}_{r,M,f}$ is  {negligible} for any $f\in G$.

\item If the action is SCC then $G$ has purely exponential growth. Further,  $\mathcal{V}_{r,M,f}$ is exponentially negligible for any $f\in G$.

\end{enumerate}
\end{prop}

\section{Conformal density on  convergence boundary}\label{SecDensity}

Let $G\curvearrowright \U$ be a proper isometric action of a non-elementary group on a proper geodesic metric space $\U$, and assume the action admits a contracting element. Fix a convergence compactification $(\pU, [\cdot])$ of $\U$. This section recalls the general setup of quasi-conformal densities supported on $\pU$, and reviews Patterson’s construction of such densities from the action $G\curvearrowright \U$.

\subsection{Patterson-Sullivan measures on convergence boundary}  
 
Let $ \mathcal M^+(\bU)$ be the set of finite positive Borel measures on $\bU:=\pU\cup\U$, on which  $G$ acts  by push-forward: $$g_\star\mu(A)=\mu(g^{-1}A)$$ 
for any Borel set $A$. Let $\mathcal{C}$ be the (non-empty) set of non-pinched points in $\pU$ in Definition \ref{ConvBdryDefn}(C).
\begin{defn}\label{ConformalDensityDefn}

Let $\omega \in [0, \infty)$. The following  map 
$$\begin{aligned}
\mu: \U \longrightarrow &\mathcal M^+(\bU) \\
x \longmapsto &\mu_x    
\end{aligned}$$
is called a 
\textit{$\omega$--dimensional $G$--quasi-equivariant quasi-conformal density}  if for any $g\in G$ and any $x\in \U$, we have
\begin{align}
\label{almostInv}
\mu_x-\textrm{ a.e. } \xi\in \pU: &\quad
\frac{dg_\star\mu_{x}}{d\mu_{x}}(\xi) \in [\frac{1}{\lambda}, \lambda],\\
\label{confDeriv}
\mu_y-\textrm{ a.e. } \xi\in \mathcal C: &\quad
\frac{1}{\lambda} e^{-\omega B_\xi (x, y)}  \le  \frac{d\mu_{x}}{d\mu_{y}}(\xi) \le \lambda e^{-\omega B_\xi (x, y)}
\end{align}
for a  universal constant $\lambda\ge 1$.  We normalize $\mu_o$ to be a probability measure: its mass $\|\mu_o\|=\mu_o(\bU)=1$.
\end{defn}

If $\lambda=1$   for (\ref{almostInv}), the map $\mu: \U \to \mathcal M^+(\bU)$ is $G$--equivariant, that is, $\mu_{gx} = g_{\star}\mu_x$; equivalently, $\mu_{gx}(gA)=\mu_x(A)$. If both inequalities hold with $\lambda=1$, we call $\mu$ a conformal density. 
\newline
\paragraph{\textbf{Patterson-Sullivan measures}}
Choose a basepoint $o \in \U$.  
The \textit{Poincar\'e series} for the proper  action  $G \act \U$ defined as follows
$$
\p_G(s, x,y) =\sum\limits_{g \in G} e^{-sd(x, gy)}, \; s \ge 0
$$
diverges at $0\le s< \e G$ and converges at $s>\e G$. The   action of $G$ on $\U$ is called \textit{divergent} if $\p_{G}(s,x,y)$ diverges at the critical exponent $\e G$. Otherwise, $G$ is called \textit{convergent}. By Proposition \ref{pro:growthtight} the SCC actions have purely exponential growth; which in particular imply they are  of divergent type.

Recall that $[\Lambda (G o)]$ is the limit set for the  action $G\act \U$, i.e. the set of accumulation points of a $G$-orbit, up to taking $[\cdot]$-saturation.

We start to construct a family of measures $\{\mu_x^{s,y}\}_{x \in \U}$ supported on $Gy$ for any given $s >\e G$ and $y\in \U$. Assume that $\p_G(s, x,y)$ is divergent at $s=\e G$. Set
\begin{equation}\label{PattersonEQ}
\mu_{x}^{s, y} = \frac{1}{\p_G(s, o, y)} \sum\limits_{g \in G} e^{-sd(x, gy)} \cdot \dirac{gy},
\end{equation}
where $s >\e G$ and $x \in \U$. Note that $\mu^{s, y}_o$ is a probability
measure supported on $Gy$. 
If $\p_G(s, x,y)$ is convergent at $s=\e G$, the Poincar\'e series in (\ref{PattersonEQ}) needs to be replaced by a modified series as in \cite{Patt}.

Fix $y\in \U$.
Choose $s_i \to \e G$ such that $\mu_x^{s_i, y}$ are convergent (in weak-* topology) in
$\mathcal M^+(\pG o)$.  The \textit{Patterson-Sullivan measures} $\mu_x^y = \lim \mu_x^{s_i,y}$ are the weak-* limit 
measures. Viewing $\mu_x^{s_i,y}$ as measures on the compactification $\bU$, any weak-* limit is supported on $\pU$. Note that $\mu_o(\pG o) = 1$. In what follows, we usually write $\mu_x=\mu_x^o$ for $x\in \U$.

\begin{thm}\label{ConformalDensityExists}
Suppose that a non-elementary $G$ acts properly on a proper geodesic space  $\U$ compactified with horofunction boundary $\hU$. Then the family $\{\mu_x:x\in\U\}$  of Patterson-Sullivan measures is  a $\e G$-dimensional $G$-equivariant conformal density supported on $[\Lambda(Go)]$.     
\end{thm}

In the sequel, we write PS-measures as shorthand for
Patterson-Sullivan measures.

\subsection{Shadow Lemma}\label{SSshadowlem}
Since $G$ is non-elementary, $G$ contains infinitely many independent contracting elements by Lemma \ref{lem:manycontractingelements}.  In what follows, we make the standing assumption:
\begin{conv}\label{ConventionF}
    Let  $F=\{f_1,f_2,f_3\}$ be a set of three mutually independent contracting elements, and consider the associated contracting system  
\begin{equation}\label{SystemFDef}
\f =\{g\cdot \ax(f_i):   g\in G, 1\le i\le 3\}
\end{equation}
Note that  the axis $\ax(f_i)$    defined in (\ref{axisdefn})  depends on the choice of a basepoint $o\in \U$ and  is $C$--contracting for some $C>0$.   
\end{conv}

Let $r>0$ and $x, y\in \U$. 
First of all, define the usual cone and shadow:  
$$\Omega_{x}(y, r)\quad :=\quad \{z\in \U: \exists [x,z]\cap B(y,r)\ne\emptyset\}$$
and $\Pi_{x}(y, r) \subseteq \pU$ be the topological closure  in $\pU$ of $\Omega_{x}(y, r)$.

The partial shadows $\Pi_o^F(go, r)$ and cones $\Omega_o^F(go, r)$   given in Definition \ref{ShadowDef} depends on the choice of a contracting system $\f$ as in (\ref{SystemFDef}). Without index $F$,   $\Pi_o(go, r)$   denotes the     usual shadow.

\begin{defn}[Partial cone and shadow]\label{ShadowDef}
For $x\in \U, y\in Go$, the \textit{$(r, F)$--cone} $\Omega_{x}^F(y, r)$ is the set of elements $z\in \U$ such that $y$ is an $(r, F)$--barrier for some geodesic $[x, z]$.  

The \textit{$(r, F)$--shadow} $\Pi_{x}^F(y, r) \subseteq \pU$ is the topological closure  in $\pU$ of the cone $\Omega_{x}^F(y, r)$ in the compactification.
\end{defn}

The key fact in the theory of conformal density is the Sullivan's  shadow lemma.
\begin{lem}\cite[Lemma 6.3]{YANG22}\label{ShadowLem}
Let $\{\mu_x\}_{x\in \U}$ be a $\omega$--dimensional $G$--equivariant conformal density on $\pU$ for some $\omega>0$. Assume that $\mu_o(\mathcal C)=1$. Then there exist $r_0,  n_0 > 0$ depending on $\f$ in Convention \ref{ConventionF} with the following property. 
For given $r \ge  r_0$, there exist $C_0=C_0(\f),C_1=C_1(\f, r)$ such that  for any $go\in Go$,
$$
\begin{array}{rl}
   C_0 \mathrm{e}^{-\omega \cdot  d(o, go)}  \le   \mu_o(\Pi_o^{F}(go,r))  \le \mu_o(\Pi_o(go,r))   \le C_1     \mathrm{e}^{-\omega \cdot  d(o, go)} 
\end{array}
$$
where we replace $F$ by $F^n:=\{f_1^n,f_2^n,f_3^n\}$ for some $n\ge n_0$. 
\end{lem}

As a corollary, we obtain a lower bound on the dimension of a conformal density.
\begin{lem}\cite[Proposition 6.8, Theorem 1.16]{YANG22}\label{UniqueConformalDensity}
Under the assumption of Lemma \ref{ShadowLem}, we have $\omega \ge \e G$.   Moreover,  if the proper action $G\act \U$ is of divergent type, then the conformal density in Theorem \ref{ConformalDensityExists} becomes unique on the quotient $[\hU]$ of the horofunction boundary by the finite difference relation.
\end{lem}

In the sequel, by taking large power of each element in $F$, we always assume that $F$ and $r\ge r_0$ satisfy Lemma \ref{ShadowLem}. We emphasize the contracting  constant $C$ of the system $\f$ is not affected under such modification of $F$.

\subsection{Conical points}
We now introduce  the    definition  of a conical point relative to the above $C$--contracting system $\f$ in Convention \ref{ConventionF}. 


\begin{defn}\label{ConicalDef2}
Let $F$ and $r\ge r_0$ be given by Lemma \ref{ShadowLem}.
A point $\xi \in \pU$ is called \textit{$(r, F)$--conical}   if for some $x\in Go$, the point $\xi$ lies in infinitely many $(r, F)$--shadows $\Pi_x^{F}(y_n, r)$ for $y_n\in Go$.  We denote by  $\Lambda_{r}^F(Go) $ the set of $(r, F)$--conical points.

\end{defn}
\begin{figure}
    \centering

\tikzset{every picture/.style={line width=0.75pt}} 

\begin{tikzpicture}[x=0.75pt,y=0.75pt,yscale=-1,xscale=1]

\draw    (100,103) .. controls (144.5,78) and (352.5,52) .. (524.5,96) ;
\draw [shift={(524.5,96)}, rotate = 194.35] [color={rgb, 255:red, 0; green, 0; blue, 0 }  ][line width=0.75]    (10.93,-3.29) .. controls (6.95,-1.4) and (3.31,-0.3) .. (0,0) .. controls (3.31,0.3) and (6.95,1.4) .. (10.93,3.29)   ;
\draw [shift={(100,103)}, rotate = 330.67] [color={rgb, 255:red, 0; green, 0; blue, 0 }  ][fill={rgb, 255:red, 0; green, 0; blue, 0 }  ][line width=0.75]      (0, 0) circle [x radius= 3.35, y radius= 3.35]   ;
\draw   (284.6,78.45) .. controls (287.68,58.59) and (316.32,46.54) .. (348.57,51.54) .. controls (380.83,56.54) and (404.48,76.69) .. (401.4,96.55) .. controls (398.32,116.41) and (369.68,128.46) .. (337.43,123.46) .. controls (305.17,118.46) and (281.52,98.31) .. (284.6,78.45) -- cycle ;
\draw    (312.5,93) -- (373.5,94) ;
\draw [shift={(373.5,94)}, rotate = 0.94] [color={rgb, 255:red, 0; green, 0; blue, 0 }  ][fill={rgb, 255:red, 0; green, 0; blue, 0 }  ][line width=0.75]      (0, 0) circle [x radius= 3.35, y radius= 3.35]   ;
\draw [shift={(312.5,93)}, rotate = 0.94] [color={rgb, 255:red, 0; green, 0; blue, 0 }  ][fill={rgb, 255:red, 0; green, 0; blue, 0 }  ][line width=0.75]      (0, 0) circle [x radius= 3.35, y radius= 3.35]   ;
\draw    (312.5,93) -- (312.5,75) ;
\draw [shift={(312.5,72)}, rotate = 90] [fill={rgb, 255:red, 0; green, 0; blue, 0 }  ][line width=0.08]  [draw opacity=0] (8.93,-4.29) -- (0,0) -- (8.93,4.29) -- cycle    ;
\draw    (373.5,94) -- (373.5,76) ;
\draw [shift={(373.5,73)}, rotate = 90] [fill={rgb, 255:red, 0; green, 0; blue, 0 }  ][line width=0.08]  [draw opacity=0] (8.93,-4.29) -- (0,0) -- (8.93,4.29) -- cycle    ;

\draw (303.28,94.91) node [anchor=north west][inner sep=0.75pt]  [rotate=-1.93]  {$g_{n} o$};
\draw (354.29,97.42) node [anchor=north west][inner sep=0.75pt]  [rotate=-1.93]  {$g_{n} f_{n} o$};
\draw (313,26.4) node [anchor=north west][inner sep=0.75pt]    {$g_{n}\mathrm{Ax}( f_{n})$};
\draw (92,108.4) node [anchor=north west][inner sep=0.75pt]    {$o$};
\draw (438,102.4) node [anchor=north west][inner sep=0.75pt]    {$\Pi _{o}^{F}( g_{n} o,r) \ni \ \xi $};
\draw (314.5,78.4) node [anchor=north west][inner sep=0.75pt]    {$\leq r$};
\draw (373.5,79.4) node [anchor=north west][inner sep=0.75pt]    {$\leq r$};

\end{tikzpicture}
    \caption{Conical points}
    \label{fig:conicpts}
\end{figure}



The following useful property \cite[Lemma 4.7]{YANG22} resembles the usual definition of conical points in Kleinian groups.
\begin{lem}\label{ConicalPointsLem}
For given $r\ge r_0$ there exists   $\hat r=\hat r(r,C)>0$ with the following property.  Let $\xi\in \Lambda_r^F(Go)$.  For any basepoint $o\in \U$ there exists a geodesic ray $\gamma$ starting at $o$ ending at $[\xi]$  with infinitely many $(\hat r, F)$-barriers.  That is, 
\begin{itemize}
    \item 
    there exist $g_n\in G$ and $f_n\in F$ so that  $d(g_no, \gamma), d(g_nf_no, \gamma)\le \hat r$.
\end{itemize} 
\end{lem}
Note, to keep $F$ unchanged, the constant $\hat r$ depends on $r$, so   $\gamma$ intersects $N_{\hat r}(Go)$ unboundedly. For the later use,  we need to derive the following uniform result. 
\begin{cor}\label{ConicalPointsIntersUnifNbhd}
Let $F$ and $r\ge r_0$ be given by Lemma \ref{ShadowLem}.
For any $\xi\in \Lambda_r^F(Go)$, any geodesic ray $\gamma$ starting at $o$ and ending at $[\xi]$ intersects $N_C(Go)$ unboundedly.    
\end{cor}
\begin{proof}
The proof is based on a general fact. Let $U$ be a $C$-contracting subset. For any $\hat r>0$ there exists $L=L(\hat r)$ so that for  a geodesic segment $\alpha$ with $\mathrm{diam}(\alpha\cap N_{\hat r}(U))>L$, we have $\alpha\cap N_C(U)\ne \emptyset$. To complete the proof, it suffices to take $F=F^n$ so that $\min\{d(o,fo): f\in F\}>L$. 
\end{proof}

Combined with Lemma \ref{lineardiverging}, we obtain   the following corollary (see \textsection \ref{subsec: horofunction boundary} for definitions).
\begin{lem}
The sublinear difference relation restricted to the set of conical points coincides with finite difference relation.
\end{lem} 
At last,  we recall the following fact saying that conical points are generic for divergence type action. It is a part of  the Hopf-Tsuji-Sullivan dichotomy \cite[Theorem 1.14]{YANG22}, where the converse (easier direction) is also true.
\begin{lem}\label{HSTLem}
Let $F$ and $r\ge r_0$ be given by Lemma \ref{ShadowLem}. Let $\{\mu_x\}_{x\in \U}$ be a $\omega$--dimensional $G$--equivariant conformal density on $\pU$ for some $\omega>0$. Assume that $\mu_o(\mathcal C)=1$.   If $G\act \U$ is of divergence type, then $\mu_o$ charges full measure on $[\Lambda_r^F(Go)]$ and every $[\xi]$--class is $\mu_o$--null.
\end{lem}


\section{Genericity of sublinearly Morse directions}\label{SecHTST}
Let $G\act \U$ be a non-elementary SCC action on a proper geodesic metric space $X$ with a contracting element. We fix a basepoint $o\in X$. This section is devoted to the proof of    \ref{MainThm} for the horofunction boundary. Actually, we prove the following   result on a general convergence boundary. 


\begin{thm}\label{HTSThm} 
Assume that $X$ is equipped with a convergence boundary $(\pU,[\cdot])$ and  $\{\mu_x\}_{x\in \U}$ is a $\omega$--dimensional $G$--quasi-equivariant quasi-conformal density  for some $\omega>0$, supported on the subset of non-pinched points $\mathcal C\subseteq \pU$.    Then the set of $[\cdot]$-classes of boundary points represented by regularly contracting rays has $\mu_o$--full measure.  
\end{thm} 

Note that regularly contracting rays accumulate into a unique $[\cdot]$-class by Lemma \ref{RegCConverge}. We start by recalling the basic setup.  

By Convention \ref{ConventionF}, we fix a set of three independent contracting elements $F\subseteq G$, so that  $C>0$ is the common contracting constant for each $f\in F$.  That is, the  axes $\ax( f)=E(f)\cdot o$ is a $C$--contracting quasi-geodesic.  For   simplicity, by taking a larger value of $C$, we may assume that  any geodesic segment with two endpoints in $\ax( f)$ is $C$--contracting. This is guaranteed by Lemma \ref{BigThree}(2) and (3).

Unless otherwise stated, assume further that $F$ and $r\ge r_0$ satisfy Lemma \ref{ShadowLem}.

\subsection{Growth tightness of elements without frequent barriers}
We introduce in Definition \ref{FreqBarrierDefn} an analogous notion of frequently contracting segments (Definition \ref{FreqContrDefn}), group elements with frequent barriers, which  involves the group action. The goal of this subsection is to prove that such elements are exponentially generic (see Proposition \ref{frequentbarriers}).

Recall that for any $\theta\in [0,1]$, a \textit{$\theta$--segment} of a geodesic segment $\gamma$ means   a connected and {closed} subsegment of $\gamma$ with length $\theta \ell(\gamma)$. If $\gamma$ has $(r, f)$--barriers, by Definition \ref{BarrierDef}, there exists $t\in G$ so that $to, tfo\in N_r(\gamma).$

\begin{defn}\label{FreqBarrierDefn}
Fix $\theta\in (0,1]$, $r>0$ and $f\in G$. We say that    a geodesic $\gamma$ contains \textit{$(r, f)$--barriers at $\theta$--frequency} if 
  every $\theta$--segment of $\gamma$ has $(r, f)$--barriers.

An element $g\in G$ has \textit{$(r, f)$--barriers at $\theta$--frequency} if there exists a geodesic $\gamma$ between $B(o, M)$ and $B(go, M)$ such that $\gamma$ has  $(r, f)$--barriers at $\theta$--frequency.
\end{defn}

We need the following two elementary lemmas. The first one is an immediate consequence  from Lemma \ref{close barriers in two geodesics}.

\begin{lem}\label{ExistsThetaBarrier}
Given $r, M>0$, there exist $\hat r=\hat r(r,C,M)$ and $L_1=L_1(C,M)>0$ with the following property. Choose $h\in E( f)$ for some $f\in F$ so that $d(o,ho)>L_1$. Let  $\alpha, \beta$ be two geodesics  from $B(o,M)$ to $B(go,M)$. If $\alpha$ contains a $\theta$--segment with $(r,  h)$--barriers, then $\beta$ contains some $\theta$--segment with  $(\hat r, h)$--barriers.      
\end{lem}
\begin{proof}
Let $p$ be a $\theta$--segment of $\alpha$ with an $(r,  h)$--barrier. That is, there exists  $t\in G$ so that  $d(to,p), d(t ho,p)\le r$. By assumption, $[o,ho]$ is a $C$-contracting segment. Let $\hat r=\hat r(r,C,M)$ be given by Lemma \ref{close barriers in two geodesics}. Thus, there exist $x,y\in \beta$ so that $d(to, x), d(tho,y)\le \hat r$.  If necessary, we may extend the segment $[x,y]$ so that it is longer to be a $\theta$--segment of $\beta$, so $\beta$ contains an $(\hat r, h)$--barrier in a $\theta$--segment. This completes the proof. 
\end{proof}

The following lemma will be used in the proof of Proposition \ref{keylemma}. Roughly speaking, every elements in $\mathcal{O}_{M}$ have  no $(r,f)$--barrier. %

\begin{lem}\label{ThinRegion}
Assuming $M\ge C$, let $\gamma$ be a geodesic with two endpoints in $N_M(Go)$ and $\alpha$ be a component  in the complement $\gamma\setminus  N_M(Go)$. Then for any $\hat r>0$, there exists a constant $L_2=L_2(\hat r,C)>0$ so that $\alpha\cap N_{\hat r}(h\ax(f))$ has diameter at most $L_2$ for any $h\in G$ and $f\in F$.   
\end{lem}
\begin{proof}
If $\alpha$ intersects $N_{\hat r}(h\ax(f))$, let $x$ and $y$ be the entry and exit points of $\alpha$ in $N_{\hat r}(h\ax(f))$ respectively; otherwise there is nothing to prove. If $x'$ and $y'$ are the corresponding projection points to $h\ax(f)$, we have $d(x,x'), d(y,y')\le \hat r$.
By assumption, $\alpha\cap N_M(Go)=\emptyset$. Noting $C\le M$ and $\ax(f)\subseteq Go$,  we have $\alpha\cap N_C(h\ax(f))=\emptyset$. The  $C$--contracting property implies the projection of $\alpha$ to $h\ax(f)$ has diameter at most $C$.  Consequently, we obtain $d(x',y')\le C$ and thus $d(x,y)\le d(x',y')+2\hat r\le C+2\hat r$. Setting  $L_2=C+2\hat r$  completes the proof.
\end{proof}
\begin{conv}\label{ConventionrM} 
From now on, we fix any $\theta\in (0,1]$ and choose the constants
\begin{enumerate}
     \item 
     $M\ge C$ to satisfy the definition of SCC actions (\ref{SCCDefn}) and Lemma \ref{ConicalPointsLem}.  
     \item 
     $r\ge M$ to satisfy Proposition \ref{pro:growthtight} and Shadow Lemma \ref{ShadowLem}. 
     \item 
     $\hat r=\hat r(r,C,M)>M$ to satisfy Lemma \ref{ExistsThetaBarrier}.
\end{enumerate}

\end{conv}

Given $1\ne h\in G$, let $\mathcal W(\theta,r,h)$ denote the set  of elements $g$ in $G$ having \textbf{no} $(r, h)$--barriers at $\theta$--frequency:  for any geodesic $\gamma$ from $B(o,M)$ to $B(go,M)$, $\gamma$ contains some $\theta$--segment without $(r, h)$--barriers.  

\begin{prop}\label{frequentbarriers}
Let   $G\act \U$ be a non-elementary proper action with a contracting element with purely exponential growth. Fix $\theta\in (0,1]$ and $r>M$ as in Convention \ref{ConventionrM}. Then for any nontrivial element $h\in G$, the set $\mathcal W(\theta,r,h)$    is   negligible.

In addition, if $G\act \U$ is a SCC action, then  $\mathcal W(\theta,r,h)$    is exponentially negligible.
\end{prop}
\begin{proof}
For simplicity, write $ W:=\mathcal W(\theta,r,h)$ in this proof.
Given $g\in  W$,  let $\gamma$ be a  fixed geodesic  from $B(o, M)$ to $B(go, M)$ so that some $\theta$--segment of $\gamma$, denoted by $\alpha$,    is $(r, h)$--barrier-free. Thus,  $\ell(\alpha)\ge \theta \ell(\gamma)$ and $\alpha$ contains no $(r, h)$--barrier. 

We may assume $\alpha$ is a maximal and closed segment with this property. We must have the two endpoints of $\alpha$ lie in $N_M(Go)$, otherwise we can extend $\alpha$ until reaching $N_M(Go)$. Thus, there exist $g_1, g_2\in G$ such that $$d(g_1o, \alpha_-), \; d(g_2o,\alpha_+)\le  M$$ and in particular,  $\hat g=g_1^{-1}g_2$ is $(r, h)$--barrier-free by Definition \ref{FreqBarrierDefn}. As a consequence, this implies that any $g\in  W$ can be written as a product of three elements $g=g_1\cdot  \hat g \cdot (g_2^{-1} g)$ so that 
 $$|d(o,go)- d(o,g_1o)-d(o,g_2o)-d(o,\hat g o)|\le 4M.$$  
By Proposition \ref{pro:growthtight},  the set of elements $\hat g\in \mathcal V_{r, M, h}$ is growth negligible: as $n\to\infty$,
$$
\frac{\sharp N(o,n)\cap \mathcal V_{r, M, h}}{\sharp N(o,n)} \to 0.
$$
If the action is SCC, then $\mathcal V_{r, M, h}$ is growth tight: for some $\epsilon>0$, 
$$
\frac{\sharp N(o,n)\cap \mathcal V_{r, M, h}}{\sharp N(o,n)} \le \mathrm{e}^{-\epsilon n}.
$$
Moreover,  the element $\hat g$ in the above product takes a definite proportion:  $$d(o, \hat go)\ge  \len(\alpha)-d(g_1o, \alpha_-)-d(g_2o, \alpha_+)\ge \ell(\alpha)-2M\ge \theta d(o,go)-2M.$$
Since $G\act \U$ has purely exponential growth, we have $\sharp  N(o,n)\asymp \mathrm{e}^{n\e G}$. By a similar argument as in \cite[Lemma 3.9]{GYANG}, a straightforward computation shows that $ W$ is growth tight if the action is SCC, and is growth negligible in general.   The lemma is then proved.
\end{proof}

\begin{cor}\label{WconvergencePSseries}
Under the assumption of Proposition \ref{frequentbarriers}, if the action is SCC, then the Poincar\'e series associated to $W:=\mathcal W(\theta,r,h)$  is convergent at $s=\e G$. Namely,
$$
\sum_{g\in W} \mathrm{e}^{-\e G d(o,go)}<\infty.
$$
\end{cor}
\begin{rem}
The convergence of the series is the key property in the proof of Theorem \ref{HTSThm} (precisely, Lemma \ref{FullFreqAtTheta}),  where  the SCC action is required. If  this convergence  holds for a proper action with purely exponential growth, we could then weaken the SCC  action to the action with purely exponential growth accordingly in Theorem \ref{HTSThm}. See the further connection to Conjecture \ref{Sullivan}.   
\end{rem}

\begin{defn}\label{defn:NFB}
Given $1\ne h\in G$, let  $\mathcal{NFB}(\theta,r, h)$  denote the $[\cdot]$-saturation of the set of  boundary points $\xi\in \pU$ that are contained in infinitely many shadows at elements in $\mathcal W(\theta,r,h)$.  Precisely, denoting $W:=\mathcal W(\theta,r,h)$, we define   
\begin{align}\label{eqn:LambdanforNFB}
\Lambda_n:=  \bigcup\limits_{k\ge n}\left(\bigcup\limits_{g \in   A(o, k,\Delta)\cap {W} }  \Pi_{o}^F(go,r)\right) 
\end{align} 
where   $A(o,n,\Delta)$ is defined in (\ref{AnnulusEQ}), and then define $$\mathcal{NFB}(\theta,r, h)= [\cap_{n\ge 1} \Lambda_n]$$ to be the union of $[\cdot]$-classes of points in $\cap_{n\ge 1} \Lambda_n$. 
\end{defn}

The next   result clarifies the set $\mathcal{NFB}(\theta,r, h)$, which shall not be used in other places though.  For simplicity, we assume the convergence boundary  is the horofunction boundary.
\begin{lem}
Given $1\ne h\in G$, let $g_n\in \mathcal W(\theta,r,h)$ be a sequence of  elements without $(r, h)$--barriers at $\theta$--frequency. Assume that $\xi$ is contained in a sequence of  shadows  $\Pi_o^F(y_n,r)$ where $y_n:=g_no$. Then there exists a geodesic ray $\gamma$ starting at $o$ and ending at $[\xi]$ such that $\gamma$ contains no  $(r, h)$--barriers at $\theta$--frequency.
\end{lem}
\begin{proof}
Let us fix $n\ge 1$ first. For $\xi\in \Pi_o^F(g_no,r)$, there exists a sequence of points $\{z_m: m\ge 1\}\subset \Omega_o^F(g_no,r)$ such that $z_m\to \xi$ and $[o,z_m]$ contains an $(r, F)$--barrier at $g_no$. Up to passage to a subsequence of $\{z_m\}$, there exists some $f'\in F$ so that
\begin{align}\label{eqn:distzm-gno}
\forall m\ge 1, \; d(g_no, [o,z_m]), d(g_nf'o, [o,z_m])\le r.    
\end{align}
This inequality remains true for any limit of $[o,z_m]$ as $m\to\infty$ by Ascoli-Arzela Lemma with respect to the local uniform convergence topology. Here we use the proper assumption of $X$.

We now vary $n\ge 1$. A Cantor diagonal argument further provides a limiting geodesic ray $\gamma$, which has an   $(r, F)$--barrier at every  $g_no$ with $n\ge 1$. As $[o,g_no]$ has no  $(r, f)$--barriers at $\theta$--frequency, any limit  $\gamma$ does so.  It remains to prove that the geodesic ray $\gamma$ ends at a point in $[\xi]$. 

To see this, first note that any limit point $\eta$ of $g_no$ lies in $[\xi]$. This follows by a {straightforward computation using (\ref{eqn:distzm-gno}) that  the difference between $b_\xi$ and $b_\eta$ is uniformly bounded by a constant depending on $r$}. As $d(g_no, \gamma)\le r$ for each $n\ge 1$, the same reasoning shows that the Busemann function associated  with $\gamma$ has also finite difference with $b_\eta$. This implies that $\gamma$ also ends at $[\eta]$, so  the proof is complete.
\end{proof}

\subsection{Non-frequently contracting segments have non-frequent barriers}

We now relate the   notion of elements with frequent barriers  to that of frequently contracting segments (Definition \ref{FreqContrDefn}). The title tells the main result, Proposition \ref{keylemma}, of  this subsection. First of all, we state a preparatory lemma.

\begin{lem}\label{fBarrierimplyfContract}
Fix some $f\in F$ and $h\in E(f)$. Let $\gamma$ be a geodesic segment with  $(r,h)$--barriers at $\theta$--frequency.  Then $\gamma$ is  $(r, C,L)$–contracting at $\theta$--frequency, where $L=d(o,ho)$.  
\end{lem}
\begin{proof}The proof follows by unravelling the definitions. By Definition \ref{FreqBarrierDefn},  if some $\theta$-segment of $\gamma$ contains an $(r,h)$--barrier, then $N_r(\gamma)$ contains two points $o,ho$ in $\ax(f)$. By the convention at the beginning of the section, $[o,ho]$ is $C$--contracting. Thus, $\gamma$ is  $(r, C,L)$–contracting at $\theta$--frequency.
\end{proof}

\begin{lem}\label{fBarrierimplyfContractRay}
Given $f\in F$, let $h_n\in E(f)$ be a sequence of elements with $d(o,h_no)\to\infty$.  Assume that  for any $\theta\in (0,1]$ and for each fixed $h_n$, $\gamma$ is a geodesic ray with  $(r, h_n)$--barriers at $\theta$--frequency. Then $\gamma$ is $(r, C)$–regularly contracting.
\end{lem}
\begin{proof}
Denoting $L_n=d(o,h_no)$, Lemma \ref{fBarrierimplyfContract} implies that $\gamma$ is $(r, C,L_n)$–contracting at $\theta$--frequency.  The conclusion follows as $\theta\in (0,1]$ is arbitrary and $L_n\to\infty$.  
\end{proof}

We now look at the set of   $(r,C,L)$--contracting rays at $\theta$--frequency. Precisely, 
\begin{defn}\label{defn:FC}
Given $L>0$, let $\mathcal{FC}(\theta,r, C,L)$  denote the set of  all $[\cdot]$--classes   $[\xi]\in [\pU]$ so that there is some geodesic ray $\gamma$ in $\U$ starting at $o$ and ending at $[\xi]$, which  is $(r,C,L)$--contracting at $\theta$--frequency.     
\end{defn}

The following proposition  shall be crucially used in the proof of  \ref{MainThm}. Before stating the result, we make precise the data.  
 

\begin{conv}  \label{ConventionfLh}
Let $r\ge M\ge C, \hat r=\hat r(r,C,M)>0$ be given as in Convention \ref{ConventionrM}. We fix some $f\in F$ and denote  $L_0=\max\{L_1,L_2\}$, where $L_1=L_1(C,M)$ and $L_2=L_2(\hat r,C)$ are given by Lemma \ref{ExistsThetaBarrier} and Lemma \ref{ThinRegion} accordingly.  
 
The connection of $f\in F$ and the constant  $C$ is that we assume any segment in $\ax(f)$ is  $C$--contracting. See the discussion at the beginning of this section.
\end{conv}
By definition, $\mathcal{FC}(\theta, r, C,L)$ and  $\mathcal{NFB}(\theta,\hat r, h)$ are both $[\cdot]$-saturated sets.
\begin{prop}\label{keylemma}
For given $L>L_0$, choose  some $h\in E(f)$ with  $d(o,ho)>2L$. Then for any $\theta\in (0,1]$,  the  set   $[\Lambda_r^F(Go)]\setminus \mathcal{FC}(\theta, r, C,L)$  is contained in   $\mathcal{NFB}(\theta,\hat r, h)$.    
\end{prop}

\begin{proof}
Given $\xi\in [\Lambda_r^F(Go)]$, let  $\gamma$ be any geodesic ray ending at $[\xi]$; that is, any unbounded sequence of points on $\gamma$ tends to some point in $[\xi]$. Abusing language we write $\gamma=[o,\xi]$. If  $\xi$ is not contained  in $\mathcal{FC}(\theta, r, C,L)$, then unveiling the definitions,  $\gamma$    is not $( r, C, L)$–contracting at $\theta$--frequency. Namely, there exist a sequence of positive numbers $R_n\to\infty$ with the following property :  
\begin{equation}
\tag{$\circledast$}\begin{aligned}
    \label{NoBarrierPrpty}
    &\text{for each } n\ge 1, \gamma[0,R_n]\text{  contains a }\theta\text{-interval }\alpha_n \text{ so that no subsegment  }\\
    &\text{of }\alpha_n \text{  with length }L\text{  is } r\text{–close to a }C\text{-contracting segment.}   
    \end{aligned}
\end{equation}
By Lemma \ref{fBarrierimplyfContract}, $\gamma[0,R_n]$ has no $( r,h)$--barriers at $\theta$--frequency. 

If the point $\gamma[R_n]$ for some $n$ lies in the closed neighborhood $N_M(Go)$, there exists $g_n\in G$ such that $d(g_no, \gamma[R_n])\le M$. By Lemma \ref{ExistsThetaBarrier}, any geodesic from $B(o,M)$ to $B(g_no,M)$ has no $(\hat r, h)$--barriers at $\theta$--frequency. This means that $g_n$ has no $(\hat r, h)$--barriers at $\theta$--frequency, so $g_n\in \mathcal{W}(\theta,\hat r, h)$ by definition. 

Otherwise, let us assume that $\gamma[R_n]$ lies outside $N_M(Go)$. The remainder of the proof is then to find another time $R_n'$ so that $\gamma[0,R_n']$ has the above property (\ref{NoBarrierPrpty}): 
\begin{itemize}
    \item it is not $( r, C, L)$–contracting at $\theta$--frequency, and the new point $\gamma[R_n']$ lies in $N_M(Go)$. 
\end{itemize} To that end, our discussion is divided into the following two cases. Write explicitly $\alpha_n=\gamma[s_n,t_n]$ for some $0\le s_n<t_n\le R_n$.

\textbf{Case 1.} Assume that  $\gamma[0,R_n]$ contains a point $\gamma[R_n']$ in $ N_M(Go)$ for some time $t_n\le R_n'<R_n$. As $\alpha_n$ is contained in $\gamma[0,R_n']$ and $R_n'\le R_n$ is an earlier time,  the segment $\gamma[0,R_n']$     has the above  property (\ref{NoBarrierPrpty}). Also,    $d(g_no, \gamma[R_n'])\le M$ for some $g_n\in G$. We are done in this case.

\textbf{Case 2.}   After $\alpha_n=\gamma[s_n,t_n]$,  $\gamma[0,R_n]$ contains no point in $N_M(Go)$. That is to say, for any $t\in [t_n, R_n]$, we have $d(\gamma[t], Go)>M$. We shall now extend    the segment $\gamma[0,R_n]$ to a future time $R_n'\ge R_n$, which  still has the    property (\ref{NoBarrierPrpty}), so that $d(g_no, \gamma[R_n'])\le M$ for some $g_n\in G$.  See Figure \ref{fig:extending}.

Indeed,   let $R_n'>R_n$ be the least number so that $\gamma[R_n']$ is contained in $N_M(Go)$. The number $R_n'<\infty$ exists, since  $\gamma$ terminates at $[\cdot]$--class of a conical point $\xi\in [\Lambda_r^F(Go)]$, which by Corollary \ref{ConicalPointsIntersUnifNbhd}  intersects  $N_{M}(Go)$ in  unboundedly many times.  

Let  $g_n\in G$ satisfy  $d(g_no, \gamma[R_n'])\le M$.  To conclude the proof in Case (2), we need to verify 
\begin{claim}
$g_n$ has no $(\hat r,h)$--barriers at $\theta$--frequency.    
\end{claim}  
\begin{proof}[Proof of the claim]
The segment $\alpha_n':=\gamma[s_n, R_n']$  is the union of $\alpha_n$ and $\gamma[t_n, R_n']$. As $\alpha_n$ is a $\theta$--interval  of $\gamma[0,R_n]$, it  follows that $\alpha_n'$ is a $\theta$--interval  of $\gamma[0,R_n']$. 

Now,  assume by contradiction that $g_n$ has  $(\hat r,h)$--barriers at $\theta$--frequency. By definition, any geodesic (e.g. say $\gamma[0, R_n']$)  between $B(o,M)$ and $B(g_no, M)$ has $(\hat r,h)$--barriers at $\theta$--frequency. In particular, the $\theta$--interval $\alpha_n'$ has $(\hat r,h)$--barriers. Hence, we have some $g\in G$ so that $go, gho$ lie in the $\hat r$--neighborhood of $\alpha_n'$. That is, $\alpha_n'\cap N_{\hat r}(g\ax(f))$ has diameter at least $d(o,ho)>2L$.
\begin{figure}[h]
    \centering

\tikzset{every picture/.style={line width=0.75pt}} 

\begin{tikzpicture}[x=0.75pt,y=0.75pt,yscale=-1,xscale=1]

\draw    (74,98) -- (163.5,97) ;
\draw [shift={(74,98)}, rotate = 359.36] [color={rgb, 255:red, 0; green, 0; blue, 0 }  ][fill={rgb, 255:red, 0; green, 0; blue, 0 }  ][line width=0.75]      (0, 0) circle [x radius= 3.35, y radius= 3.35]   ;
\draw    (163.5,97) -- (260.5,95) ;
\draw [shift={(260.5,95)}, rotate = 358.82] [color={rgb, 255:red, 0; green, 0; blue, 0 }  ][fill={rgb, 255:red, 0; green, 0; blue, 0 }  ][line width=0.75]      (0, 0) circle [x radius= 3.35, y radius= 3.35]   ;
\draw [shift={(163.5,97)}, rotate = 358.82] [color={rgb, 255:red, 0; green, 0; blue, 0 }  ][fill={rgb, 255:red, 0; green, 0; blue, 0 }  ][line width=0.75]      (0, 0) circle [x radius= 3.35, y radius= 3.35]   ;
\draw  [line width=0.75]  (357.5,87) .. controls (357.45,82.33) and (355.1,80.02) .. (350.43,80.07) -- (273.62,80.86) .. controls (266.95,80.93) and (263.6,78.64) .. (263.55,73.97) .. controls (263.6,78.64) and (260.29,81) .. (253.62,81.07)(256.62,81.04) -- (170.43,81.93) .. controls (165.76,81.98) and (163.45,84.33) .. (163.5,89) ;
\draw   (54.5,91) .. controls (71.5,50) and (183.5,156) .. (258.5,147) .. controls (333.5,138) and (370.5,57) .. (398.5,100) .. controls (426.5,143) and (319,178) .. (228,173) .. controls (137,168) and (37.5,132) .. (54.5,91) -- cycle ;
\draw    (260.5,95) -- (308.5,95) ;
\draw    (357.5,94) -- (435.5,94) ;
\draw [shift={(438.5,94)}, rotate = 180] [fill={rgb, 255:red, 0; green, 0; blue, 0 }  ][line width=0.08]  [draw opacity=0] (8.93,-4.29) -- (0,0) -- (8.93,4.29) -- cycle    ;
\draw [shift={(357.5,94)}, rotate = 0] [color={rgb, 255:red, 0; green, 0; blue, 0 }  ][fill={rgb, 255:red, 0; green, 0; blue, 0 }  ][line width=0.75]      (0, 0) circle [x radius= 3.35, y radius= 3.35]   ;
\draw    (308.5,95) -- (357.5,94) ;
\draw [shift={(308.5,95)}, rotate = 358.83] [color={rgb, 255:red, 0; green, 0; blue, 0 }  ][fill={rgb, 255:red, 0; green, 0; blue, 0 }  ][line width=0.75]      (0, 0) circle [x radius= 3.35, y radius= 3.35]   ;
\draw    (359,123) -- (357.67,99) ;
\draw [shift={(357.5,96)}, rotate = 86.82] [fill={rgb, 255:red, 0; green, 0; blue, 0 }  ][line width=0.08]  [draw opacity=0] (8.93,-4.29) -- (0,0) -- (8.93,4.29) -- cycle    ;
\draw [shift={(359,123)}, rotate = 266.82] [color={rgb, 255:red, 0; green, 0; blue, 0 }  ][fill={rgb, 255:red, 0; green, 0; blue, 0 }  ][line width=0.75]      (0, 0) circle [x radius= 3.35, y radius= 3.35]   ;
\draw  [line width=0.75]  (163.5,103) .. controls (163.6,107.67) and (165.98,109.95) .. (170.65,109.85) -- (204.35,109.13) .. controls (211.02,108.99) and (214.4,111.25) .. (214.5,115.92) .. controls (214.4,111.25) and (217.68,108.85) .. (224.35,108.71)(221.35,108.77) -- (250.65,108.15) .. controls (255.32,108.05) and (257.6,105.67) .. (257.5,101) ;

\draw (146,99.4) node [anchor=north west][inner sep=0.75pt]    {$s_{n}$};
\draw (261.5,98.4) node [anchor=north west][inner sep=0.75pt]    {$t_{n}$};
\draw (301,98.4) node [anchor=north west][inner sep=0.75pt]    {$R_{n}$};
\draw (348,58.4) node [anchor=north west][inner sep=0.75pt]    {$R'_{n}$};
\draw (184.29,116.12) node [anchor=north west][inner sep=0.75pt]    {$\ell ( \alpha _{n}) =\theta \cdot R_{n}$};
\draw (445,100.4) node [anchor=north west][inner sep=0.75pt]    {$\xi $};
\draw (219,148.4) node [anchor=north west][inner sep=0.75pt]    {$N_{M}( Go)$};
\draw (348,129.4) node [anchor=north west][inner sep=0.75pt]    {$h_{n} o$};
\draw (69,101.4) node [anchor=north west][inner sep=0.75pt]    {$o$};
\draw (258,54.4) node [anchor=north west][inner sep=0.75pt]    {$\alpha '_{n}$};

\end{tikzpicture}
    \caption{Case (2) in the proof of Proposition \ref{keylemma}}
    \label{fig:extending}
\end{figure}

As  $\gamma[t_n, R_n']$  lies outside $N_M(Go)$, we obtain from Lemma \ref{ThinRegion} that $\gamma[t_n, R_n']\cap N_{\hat r}(g\ax(f))$ has diameter at most $L_2=L_2(\hat r,C)$. Noting $\alpha_n'=\alpha_n\cup \gamma[t_n, R_n']$, $\alpha_n\cap N_{\hat r}(g\ax(f))$ has diameter at least $2L-L_2>L$. By the assumption before Definition \ref{FreqBarrierDefn}, since any segment with two endpoints in $g\ax(f)$ is $C$--contracting, we see that $\alpha_n$ contains a segment of length $L$, that is $\hat r$–close to a $C$--contracting segment. This  contradicts  (\ref{NoBarrierPrpty}), so we proved that $g_n\in \mathcal{W}(\theta,\hat r,h)$.
\end{proof}

In summary, we produced a sequence of elements $g_n\in \mathcal{W}(\theta,\hat r,h)$ so that $d(g_no, \gamma)\le M\le \hat r$. The geodesic $\gamma=[o,\xi]$ terminates at the $[\cdot]$-class $[\xi]$. We then have $\xi\in [\Pi_{o}^F(g_no,\hat r)]$. Hence,  $\xi$ lies in $\mathcal{NFB}(\theta,\hat r, h)$ according to Definition \ref{defn:NFB}.  The proposition is proved.
\end{proof}

In order to prove that $\mathcal{FC}(\theta,r, C,L)$ is $\mu_o$--full, by Proposition \ref{keylemma} we only need to prove that $\mathcal{NFB}(\theta,\hat r, h)$ is $\mu_o$--null. 
The latter relies on Proposition \ref{frequentbarriers}.

\begin{lem}\label{FullFreqAtTheta}
Assume that the action $G\act \U$ is SCC. 
Fix  a nontrivial element  $h\in E(f)$.
Then for any $\theta\in (0,1)$, the set $\mathcal{NFB}(\theta,\hat r, h)$ is $\mu_o$--null.  
\end{lem}
\begin{proof}
In this proof, we write $W:=\mathcal{W}({\theta,\hat r,h}),\; \mathcal{NFB}:=\mathcal{NFB}({ \theta,\hat r, h})$ for simplicity. The following Borel-Cantelli argument is standard, but we write it for completeness. 

By way of contradiction,  assume that $\mu_o(\mathcal{NFB})>0$.  Then  for all $n\gg 0$,  $$\mu_o(\Lambda_n)\ge \frac{1}{2}\mu_o(\mathcal{NFB})>0$$ where the set $\Lambda_n$ is defined in (\ref{eqn:LambdanforNFB}). Note that $\omega\ge \e G$ by Lemma \ref{UniqueConformalDensity}. Applying Proposition \ref{frequentbarriers} (with $r=\hat r$), $W\subseteq G$ is exponentially negligible, so by Corollary \ref{WconvergencePSseries} the following Poincar\'e series  associated with $W$ is convergent:
$$
\p_{W}(\omega, o,o)=\sum_{g\in W} e^{-\omega d(o,go)}<\infty 
$$

By  Lemma    \ref{ShadowLem} we have  $e^{-\omega d(o,go)} \asymp \mu_o(\Pi_o(go,\hat r))$ for each $g\in G$. Note that  each $go$ is contained in at most $(2\Delta)$   annulus  $A(o, k,\Delta)$ with $|d(o,go)-k|\le \Delta$. Consequently,   we obtain a uniform lower bound on the partial sum  of   $\p_{{W}}(\omega,o,o)$ for any $n\gg 0$:
$$
\sum\limits_{k\ge n}\left(\sum\limits_{v \in   A(o, k,\Delta)\cap {W} }    e^{-\omega k}\right) \succ \mu_o(\Lambda_n)
$$ thus  contradicting   the convergence of $\p_{W}(\omega,o,o)$. This shows that $\mathcal{NFB}$ is $\mu_o$--null.  
\end{proof}

\subsection{Completion of the proof of Theorem \ref{HTSThm}}
Recall that the constants  $r\ge M\ge C, \hat r=\hat r(r,C,M), L_0$ and $h\in E(f)$ are given in Convention \ref{ConventionfLh}. 

We apply Proposition \ref{RCasCountableIntersection} (with $r:=\hat r$), which provides the constant  $\tilde r=\tilde r(\hat r)>0$ with the following relation:
$$
\mathcal{RC}(\tilde r, C)  \supset \bigcap_{L_0\le L\in \mathbb N}\left (\bigcap_{\theta\in (0,1]\cap \mathbb Q} \mathcal{FC}(\theta,\hat r, C,L)\right)$$
Note that  $\mathcal{RC}(\tilde r, C)$ comprise the $[\cdot]$-classes of boundary points in $\pU$ represented by $(\tilde r, C)$-regularly contracting rays. 

According to Proposition \ref{keylemma}, for each $L\gg L_0$, we may choose  some $h\in E(f)$ so that $\mathcal{FC}(\theta,r, C,L)$ contains the  set  $[\Lambda_r^F(Go)]\setminus \mathcal{NFB}(\theta,\hat  r, h)$. Note that $[\Lambda_r^F(Go)]$ is $\mu_o$-full by Lemma \ref{HSTLem}, and    $\mathcal{NFB}(\theta,\hat r, h)$ is $\mu_o$-null by Lemma \ref{FullFreqAtTheta}. Thus,  for every large $L>1$ and $\theta\in (0,1]$, $\mathcal{FC}(\theta,\hat r, C,L)$ it is $\mu_o$--full. Hence,  their countable intersection  is still $\mu_o$--full, so $\mathcal {RC}(\tilde r,C)$ is $\mu_o$--full.  The proof is complete.

\begin{proof}[Proof of \ref{MainThm}]
The horofunction boundary $\hU$ is a convergence boundary for $X$ by Theorem \ref{HorobdryConvergence} where all boundary points are non-pinched. By Theorem \ref{ConformalDensityExists}, there exists a   $\e G$-dimensional conformal density on  $\mathcal C= \hU$, which is unique in the sense of Theorem \ref{UniqueConformalDensity}. So we can apply Theorem \ref{HTSThm}.

By the above proof, $\mathcal{RC}(\tilde r, C)$ comprise the $[\cdot]$-classes of boundary points in $\hU$ represented by $(\tilde r, C)$-regularly contracting rays. By Theorem~\ref{GQRtheorem}, the equivalent classes of regularly contracting rays  form a subset of the sublinearly Morse boundary. Hence, $\mathcal{RC}(\tilde r, C)$  is a subset of sublinear Morse directions and   \ref{MainThm}  follows  from Theorem \ref{HTSThm}. 
\end{proof}

\bibliographystyle{amsplain}
 \bibliography{bibliography}

\end{document}